\documentclass{SCIS2021}

\usepackage[english]{babel}
\usepackage{amsthm}
\usepackage{graphicx}
\usepackage{amssymb}
\usepackage{algorithm}

\usepackage{algorithmic}
\usepackage{verbatim}
\usepackage{arydshln}
\usepackage{xcolor}
\usepackage[normalem]{ulem}
\usepackage{bm}      
\usepackage{mathrsfs} 
\usepackage{verbatim}

\newtheorem{Assumption}{Assumption}


\begin{document}
	\ArticleType{RESEARCH PAPER}
	\Year{2021}
	\Month{}
	\Vol{}
	\No{}
	\DOI{}
	\ArtNo{}
	\ReceiveDate{}
	\ReviseDate{}
	\AcceptDate{}
	\OnlineDate{}
	
	\title{Optimal intrinsic formation using exogenous systems}{Optimal intrinsic formation using exogenous systems}
	
	\author[1,2]{Yueyue XU}{}
        \author[2]{Panpan ZHOU}{}
        \author[1]{Lin WANG}{{wanglin@sjtu.edu.cn}}
        \author[2]{Xiaoming HU}{}
	\AuthorMark{Xu Y}
	
		\AuthorCitation{Xu Y, Zhou P, Wang L, et al}
	
	

	\address[1]{Department of Automation, Shanghai Jiao Tong University, Shanghai {\rm 200240}, China}
        \address[2]{Department of Mathematics, KTH Royal Institute of Technology, Stockholm {\rm 10044}, Sweden}
	
	\abstract{This paper investigates the intrinsic formation problem of a multi-agent system using an exogenous system. The problem is formulated as an intrinsic infinite time-horizon linear quadratic optimal control problem, namely, no formation error information is incorporated in the performance index.
Convergence to the formation is achieved by utilizing an exogenous system, thus expanding the steady-state formation space of the system. 
For the forward problem, we provide the existence condition for a nonzero steady state and characterize the steady-state space. For the inverse problem, we design both the input matrix and the exogenous system so that the desired formation can be achieved. Finally, numerical simulations are provided to illustrate the effectiveness of the proposed results. }
	\keywords{multi-agent system, intrinsic formation control, exogenous system, quadratic optimal control, linear system}
	
	\maketitle


\section{Introduction}
The formation control problem has been an interesting and attractive topic in the field of multi-agent coordination control for its great applications in, for instance, surveillance or exploration activity, transportation, and space missions  \cite{zhang2025linear,li2022differential}. The formation control problem aims to design cooperative control strategies to steer a group of agents to form a desired geometric pattern. 

In most of the existing literature on formation control, multi-agent systems range from single or double integrators \cite{zhao2019bearing,vu2023distance}, to more general linear systems \cite{dong2016time}, and even to nonlinear systems such as unicycles \cite{roza2019smooth} and Euler-Lagrange systems \cite{zhou2020formation}. Among the formation control approaches, displacement-based method is one of the fundamental methods \cite{oh2015survey}. As far as we know, in the majority of research results that utilize the displacement-based method, the desired formation configuration is included in the designed control strategy, no matter if they formulate the problem as an optimization problem \cite{wang2012integrated,chang2022optimal,liu2023formation} or not \cite{dong2016time,coogan2012scaling}. The explicit desired formation is easy to be stolen, which causes security problem of the formation control. In recent years, some researchers have begun studying the formation control problem using an intrinsic approach \cite{song2017intrinsic,zhang2018intrinsic,zhang2020intrinsic,li2022differential,li2024Pattern}, which does not require including any information about the desired formation in the controller or the performance index. Instead, the resulting formation pattern is determined by the specified inter-agent connectivity topology. 
However, the existing literature has significant limitations in terms of formation diversity. To the best of our knowledge, the intrinsic approach was firstly used in \cite{song2017intrinsic} to study the formation control, where agents converge to a great circle.
Later, Zhang et al. \cite{zhang2018intrinsic} used gain functions to achieve intrinsic formation control for reduced attitude without any prior information of the desired formation, but it was limited to achieving a tetrahedron formation. After that, Zhang et al. \cite{zhang2020intrinsic} extended this work to achieve all regular polyhedra formations. In \cite{li2022differential}, a differential game approach is used to achieve intrinsic control. But due to the symmetry requirements of the objective function for the game, the formations were restricted to Platonic solids. It is also worth noting that the above work used static feedback control, while we study the dynamic feedback control in this paper by adding an exogeneous system. Even though full state information is required, more formation configurations can be achieved.

One of the contributions of our model is to expand the range of formations achievable through intrinsic control, enabling essentially arbitrary formations in any spatial dimension—excluding only certain degenerate cases (e.g., those associated with eigenvectors of matrix $A$ having eigenvalues with positive real parts or zero eigenvalues with nontrivial Jordan blocks).
We formulate the formation control problem as an infinite time-horizon linear quadratic optimal control problem. It is widely understood that, if the performance index does not incorporate the difference between the agent's current position and the target position, the optimal trajectory derived from a standard linear quadratic optimal control problem can at most converge to the null space of the system matrix. However, in many cases, the desired formation does not lie within that space. Thus we expand the steady-state formation space of the system by incorporating an exogenous system. Although Li et al. \cite{li2024Pattern} also studied the formation using an intrinsic optimal control approach, it only added a single integrator to the system and the control strategies can only achieve two-color patterns in a two-dimensional space. Plus, we can control only a subset of the players, relying on their inherent interactions to guide the entire system to the desired formation.

In this paper, we firstly address the forward problem, which includes determining the existence condition for a nonzero steady state and characterizing the steady-state space. We then investigate the inverse problem which aims to design the system to ensure convergence to the desired formation. To achieve this, we design the input matrix so that the desired formation lies within the maximum steady-state space while minimizing the number of control. Then, we design the exogenous system and performance index to ensure convergence to the specific desired formation. Finally, we specify how to set the initial state of the exogenous system so that any desired formation scaling can be achieved.

The designed exogenous system can protect the formation information in three ways. 
Firstly, the desired formation pattern is difficult to be stolen since it does not exist in the controller and the performance index explicitly \cite{yang2021attacks}.
Secondly, the same formation pattern can be achieved with different controls corresponding to different exogenous systems, which helps enhancing the stealth of the desired formation.
Thirdly, instead of minimizing the input of the multi-agent system, we minimize the input of the exogenous system.
Thus the agents do not travel from the start position to the desired position via the shortest route; instead, they take alternative paths to confuse any adversaries that may attempt to seize the desired formation.

This paper is organized as follows. In
Section 2, we first formulate the formation problem under the framework of intrinsic quadratic optimal control. 
Then we study the existence condition of a  nonzero steady state and characterize the steady-state space in Section 3. Section 4 presents the design of the input matrix, exogenous system, performance index and initial state required to achieve the desired formation. In Section 5, simulations are given to illustrate the theoretical results. A summary of notations is provided in Table 1.
\begin{table}
\begin{center}
\caption{Notations}
\vskip0.1cm
\label{tab1}
\begin{tabular}{c p{5cm}}
\hline
Notations & Definitions \\
\hline
$\operatorname{Im}(A)$ & Image space of matrix $A$ \\
$\operatorname{Ker}(A)$ & Kernel space of matrix $A$ \\
 $\operatorname{rank}(A)$ & The rank of matrix $A$ \\
$I_{n}$ & $n\times n$ identity matrix  \\
$\operatorname{diag}(X_1,...,X_n)$ & Block diagonal matrix with blocks $X_1,...,X_n$\\ 
$\otimes$  & Kronecker product\\
\hline
\end{tabular}
\end{center}
\end{table}
\section{Problem formulation}
Consider a multi-agent system with $N$ agents, $x_i(t) \in \mathbb{R}^d$ is the state of the agent $i$, where $d$ is the dimension of space. We can arrange the states of $N$ agents into a column vector $x(t) \in \mathbb{R}^{n}$ where $n=N\times d$. Then we write the dynamics of multi-agent system in the compact form
\begin{equation}\label{eq1}
\dot{x}(t)= A x(t)+B u(t),
\end{equation}
where $u(t)\in \mathbb{R}^m$ is the control for the multi-agent system and the form of matrix $A$ depends on the interaction patterns among the agents. For example, in Laplacian networked system, the dynamics of the agent $i$ are described by $\dot{x}_{i}(t)=\sum_{j \in \mathcal{N}_{i}}\left(x_{j}(t)-x_{i}(t)\right)+B_iu(t),$
where $\mathcal{N}_{i}$ is the set of agents that interact with the agent $i$, $B_i \in \mathbb{R}^{d \times m}$ is the input matrix. Then we have $A=-L\otimes I_{d},~ B=\left[B_{1}^{\top},~ B_{2}^{\top},~ \cdots ,~B_{N}^{\top}\right]^{\top}$, where $L$ is the Laplacian matrix. 

As mentioned above our goal is to solve the formation problem in an intrinsic way without incorporating any desired formation information in the performance index.  Then a standard way to formulate the problem under the framework of quadratic optimal control would be as follows:
\begin{equation}\label{origin_op}
\begin{aligned}
\min_u \quad & J = \frac{1}{2} \int_0^{\infty} \left(x^\top Q x+ \|u\|^{2}  \right) dt, \\
s.t. \quad & \dot{x} = Ax + Bu,
\end{aligned}
\end{equation}
where the positive semi-definite matrix $Q$ is to be designed. 

However, directly minimizing the index $J$ will drive the state at most towards the null space of the matrix $A$, which may not include the formations of interest. In paper \cite{li2024Pattern}, an integrator is added to the system to modify the control signal, and by this way some stripe patterns are achieved. However, the patterns achieved there are limited to two-color patterns on a two-dimensional plane. In this paper, we enable more general formations in spaces of any dimension by introducing a general exogenous system:
\begin{equation}\label{eq4}
\begin{array}{cl}
\dot{w}=&F_{1} x+F_{2} w+G v,\\
u=&H w+K v, 
\end{array}
\end{equation}
where  $w \in \mathbb{R}^{k}$ is the state of the exogenous system, $F_{1} \in \mathbb{R}^{k \times n},F_{2} \in \mathbb{R}^{k \times  k}, G\in \mathbb{R}^{k \times  p},H \in \mathbb{R}^{m \times k}, K \in \mathbb{R}^{m \times  p}$ are matrices of the exogenous system to be designed, $v \in \mathbb{R}^{p} $ is the new control that substitutes $u$ in the performance index of \eqref{origin_op} and the performance index becomes
\begin{equation}\label{eq5}
\bar{J}=\frac{1}{2} \int_{0}^{\infty}\left(x^{\top} Q x+\|v\|^{2}\right) d t. 
\end{equation}
The above optimization problem \eqref{eq1}, \eqref{eq4}-\eqref{eq5} can be restated as the following linear-quadratic optimal control problem,
\begin{subequations}\label{eq_optimal}
\begin{align}\label{perfor_index}
\min _{v}~ & \bar{J}=\frac{1}{2} \int_{0}^{\infty}\left(\bar{x}^{\top} \bar{Q} \bar{x}+\|v\|^{2}\right) d t,\\
\text { s.t. } & \label{combined_sys} \dot{\bar{x}}=\bar{A}\bar{x}+\bar{B} v,
\end{align}
\end{subequations}
where 
\begin{equation}\label{matrix}
\begin{array}{l}
\bar{x}=\left[\begin{array}{l}
x \\
w
\end{array}\right], \bar{A}=\left[\begin{array}{cc}
A & B H \\
F_{1} & F_{2}
\end{array}\right], \\
\bar{B}=\left[\begin{array}{c}
BK \\
G
\end{array}\right], \bar{Q}=\left[\begin{array}{cc}
Q & 0_{n \times k} \\
0_{k \times n} & 0_{k \times k}
\end{array}\right].
\end{array}
\end{equation}
$Q$ is positive semi-definite and can be rewritten as $Q=C^\top C$,
where $C \in \mathbb{R}^{q \times n}.$ Define $\bar{C}=[C \quad 0_{q\times k}],$ then $\bar{Q}=\bar{C}^\top\bar{C}$.
If the optimal control exists, then it can be expressed as
\begin{equation}\label{eq9}
v^{*}=-\bar{B}^{\top} P \bar{x}, 
\end{equation}
where $P \in \mathbb{R}^{(n+k) \times (n+k)}$ is the smallest real symmetric positive semi-definite solution to the algebraic Riccati equation
\begin{equation}\label{Rieq}
\bar{A}^{\top} P+P \bar{A}-P \bar{B}\bar{B}^{\top} P+\bar{Q}=0.
\end{equation}
We say that $P$ is the smallest if for any other real symmetric positive semi-definite solution $P_0$ of \eqref{Rieq}, the difference $P_0-P$ is positive semi-definite.
The closed-loop system under the control strategy \eqref{eq9} is as follows:
\begin{equation}\label{closeloop}
\dot{\bar{x}} =(\bar{A}-\bar{B}\bar{B}^{\top} P)\bar{x}\triangleq \bar{A}_{cl} \bar{x}.
\end{equation}

\section{Forward problem: steady state description}
In this section, we will address the forward problem, namely describing the steady-state formation as solution to the optimal control problem \eqref{eq_optimal}. Note that since the formations we are interested are in general non-consensus formations, $(\bar{C}, \bar{A})$ can not be detectable, which differs from the standard case.
\subsection{Existence condition of a nonzero steady state}
In order to give the existence condition of a nonzero steady state, we first cite the following lemma that provides the existence condition of an optimal control for \eqref{eq_optimal}.
\begin{lemma}\label{the1}(cf. \cite{kuvcera1973review,Trentelman2001})
Suppose $(\bar{A}, \bar{B})$ is stabilizable and $\bar{Q}$ is positive semi-definite, then there exists a smallest real symmetric positive semi-definite solution $P$ for the algebraic Riccati equation \eqref{Rieq}. Furthermore,
\begin{equation}
v(t) = -\bar{B}^\top P \bar{x}(t)
\end{equation}
is the optimal solution of the linear quadratic optimal control problem \eqref{eq_optimal}. The closed-loop matrix $\bar{A}_{cl}$ is stable if and only if $(\bar{C},\bar{A})$ is detectable.
\end{lemma} 
By Lemma \ref{the1}, to ensure the existence of the optimal control, we assume that 
\begin{Assumption}\label{assump2}
The pair $(\bar{A}, \bar{B})$ is stabilizable.
\end{Assumption}
Besides, in order to assure that the optimal trajectory converges to a nonzero stationary steady state, we make the following assumption.
\begin{Assumption}\label{assump1}
The pair $(\bar{C},\bar{A})$ is undetectable and the only undetectable eigenvalue of $\bar{A}$ for the pair $(\bar{C},\bar{A})$ is $0$, and its algebraic multiplicity and geometric multiplicity are equal.   
\end{Assumption}
This assumption is based on the fact that undetectable eigenvalues cannot be stabilized by optimal control, and it guarantees that under optimal control, the closed-loop system \eqref{closeloop} is marginally stable.
That is, all eigenvalues of \( \bar{A}_{cl} \) are non-positive, and every Jordan block corresponding to the eigenvalue 0 is of size \( 1 \times 1 \).
\subsection{Steady-state space}
Denote the steady state of the multi-agent system and the exogenous system by $x_{ss}$ and $w_{ss}$ respectively. Set the combined state $\bar{x}_{ss}=\begin{bmatrix} x_{ss} \\ w_{ss} \end{bmatrix} $, then  the following conditions hold:
\\ (\uppercase\expandafter{\romannumeral 1})  $P \bar{x}_{ss}= 0,$ 
\\(\uppercase\expandafter{\romannumeral 2}) $\bar{Q} \bar{x}_{ss}= 0$, i.e., $\bar{C}\bar{x}_{ss}= 0$,
\\ (\uppercase\expandafter{\romannumeral 3})  $\bar{A} \bar{x}_{ss}= 0.$ 

Condition (\uppercase\expandafter{\romannumeral 1}) is implied by the existence of optimal control while (\uppercase\expandafter{\romannumeral 2}) can be derived from the fact that $\operatorname{Ker}(P)\subseteq \operatorname{Ker}(\bar C)$.
Condition (\uppercase\expandafter{\romannumeral 3}) ensures that the steady state is stationary. 

Although (\uppercase\expandafter{\romannumeral 1}) is well known, it does not seem to be easy to locate a proof. For the sake of completeness, we provide a brief proof as follows. 
Since the optimal solution of the linear quadratic optimal control problem \eqref{eq_optimal} exists, the performance index in \eqref{eq_optimal} is bounded for all feasible control. Consider the finite horizon index
\begin{equation}\label{finite_integral}
\bar{J}_T =\frac{1}{2} \int_{t}^{t+T}\left(\bar{x}^{\top} \bar{Q} \bar{x}+\|v\|^{2}\right) d t, 
\end{equation}
then the optimal cost of \eqref{finite_integral} is $\bar{x}^{\top}(t) P(T) \bar{x}(t)$. Thus $\bar{J}_T \geq \bar{x}^{\top}(t) P(T) \bar{x}(t)$ for any feasible control. Since  $\forall\,T>0, \bar{J}_T \to 0$ when $t \to \infty$, otherwise the infinite integral will not converge, we have 
$\lim_{t \to \infty} \bar{x}^{\top}(t) P(T) \bar{x}(t)= 0.$ Thus $$\lim_{T \to \infty} \lim_{t \to \infty} \bar{x}^{\top}(t) P(T) \bar{x}(t) = \lim_{t \to \infty} \bar{x}^{\top}(t) P\bar{x}(t)= 0,$$
where $P$ is the smallest solution of the algebraic Riccati equation (\ref{Rieq}). Thus $\lim_{t \to \infty}P\bar{x}(t)=P \bar{x}_{ss}= 0$, which completes the proof.

The following result gives the expression of the steady state $\bar{x}_{ss}$.
\begin{lemma}\label{theorem1} (Steady state).
Under Assumptions \ref{assump2}-\ref{assump1}, the state of the closed-loop system \eqref{closeloop} will converge to a steady state \(\bar{x}_{ss}\) which satisfies
\begin{align}
\bar{x}_{ss}= \sum_{i=1}^{s} (\phi_i^\top \bar{x}(0)) \phi_i ,
\end{align}
where $\phi_i \in \mathbb{R}^{(n+k) \times 1}~(i=1,2,\cdots,s)$ are orthonormal eigenvectors of the zero eigenvalue of the closed-loop matrix $\bar{A}_{cl}$, $s$ is the algebraic multiplicity of the zero-eigenvalue.
\end{lemma}
\begin{proof}
Based on Assumptions \ref{assump2}-\ref{assump1}, $\bar{A}_{cl}$ is marginally stable, meaning that each zero eigenvalue is associated with a first-order Jordan block, and all other eigenvalues have negative real parts.
Let $\operatorname{diag}(J_{cl},0_{s \times s}) = V^{-1} \bar{A}_{cl} V$ denote the Jordan canonical form of $\bar{A}_{cl}$. Here, $J_{cl}$ comprises all Jordan blocks corresponding to eigenvalues with negative real parts, while \(0_{s \times s}\) is an \(s \times s\) zero matrix representing the Jordan blocks associated with zero eigenvalues.
For the close-loop system \eqref{closeloop}, straight calculation gives 
\begin{align}
&\bar{x}_{ss} \!
=\! \lim_{t\rightarrow \infty} e^{\bar{A}_{cl}t}\bar{x}(0)
= \!\lim_{t\rightarrow \infty}V\, e^{\operatorname{diag}(J_{cl},0)t}V^{-1}\bar{x}(0) \nonumber\\[2mm]
&=\!\lim_{t\rightarrow \infty} V\operatorname{diag}(e^{J_{cl}t},I_s)V^{-1}\bar{x}(0)\!=\!V\operatorname{diag}(0,I_s)V^{-1}\bar{x}(0).
\end{align}
By standard construction of the $V$ matrix, we then have
$\bar{x}_{ss} \!=\! \sum_{i=1}^{s} \phi_i  \phi_i^\top \bar{x}(0)\!=\!\sum_{i=1}^{s} (\phi_i^\top \bar{x}(0)) \phi_i$, and the proof is complete.
\end{proof}
Here $\phi_{i}^{\top} \bar{x}(0)\in \mathbb{R}$ represents the weight of $\phi_{i}$ in $\bar{x}_{ss}$ and the steady state $\bar{x}_{ss}$ is a linear combination of the right eigenvectors of the closed-loop matrix $\bar{A}_{cl}$ corresponding to the zero eigenvalue, i.e., $\bar{x}_{ss} \in \operatorname{Ker}(\bar{A}_{cl})$. 

Actually, $\operatorname{Ker}(\bar{A}_{cl})$ is the \textbf{steady state space}. For \eqref{closeloop}, if $ \bar{x}(0)\in \operatorname{Ker}(\bar{A}_{cl}) $, there is $ \bar{x}(t) =\bar{x}(0),\forall t \ge 0 $. Therefore, $ \forall \phi \in \operatorname{Ker}(\bar{A}_{cl}) $, $\phi$ must be a steady state and satisfies conditions (I)-(III), i.e.,
\begin{equation}
\operatorname{Ker}(\bar{A}_{cl})
\;\subseteq\;
\operatorname{Ker}(\!\begin{bmatrix}
P \\
\bar{C}\\
\bar{A}
\end{bmatrix})
\;=\;
\operatorname{Ker}(\begin{bmatrix}
P \\ \bar{A}
\end{bmatrix}),
\end{equation}
where $ \operatorname{Ker}(P) \subseteq \operatorname{Ker}(\bar{C}) $ is used.

Because $\operatorname{Ker}(\bar{A}_{cl})$ is related to the solution of Riccati equation, it is hard to analyze. 
Then we define another space $\operatorname{Ker}( \left[\begin{array}{ll}
\bar{A}^\top \,\, \bar{C}^\top
\end{array}\right]^\top)$.    
We will prove the equivalence of the steady-state space $\operatorname{Ker}(\bar{A}_{cl})$ and $\operatorname{Ker}( \left[\begin{array}{ll} \bar{A}^\top \,\, \bar{C}^\top \end{array}\right]^\top)$ as follows.
\begin{proposition}\label{theorem0}(Steady-state space).
Consider the optimal control problem \eqref{eq_optimal} under Assumptions \ref{assump2}-\ref{assump1}. Then the steady-state space $\operatorname{Ker}(\bar{A}_{cl})$ satisfies
\begin{equation}\label{steady_space}
\operatorname{Ker}(\bar{A}_{cl})=\operatorname{Ker}( \left[\begin{array}{ll} \bar{A}^\top \,\, \bar{C}^\top \end{array}\right]^\top).
\end{equation}
\end{proposition}
\begin{proof}
From Conditions (II) and (III) for a steady state, we immediately see that $\operatorname{Ker}(\bar{A}_{cl}) \subseteq \operatorname{Ker}( \left[\begin{array}{ll} \bar{A}^\top \,\, \bar{C}^\top\!\end{array}\right]\!^\top)$. 

Next, we prove the reverse inclusion. $\forall \bar{x} \in \operatorname{Ker}( \left[\begin{array}{ll} \bar{A}^\top \,\, \bar{C}^\top \end{array}\right]^\top)$, we have $\bar{Q}\bar{x}=0.$ 
By left-multiplying the algebraic Riccati equation \eqref{Rieq} by $\bar{x}^{T}$ and right-multiplying by $\bar{x}$, we can obtain $\bar{B}^\top P \bar{x}= 0$. Thus $\bar{A}_{c l}\bar{x}=(\bar{A}-\bar{B}\bar{B}^{\top}P)\bar{x}=0$, which implies $\operatorname{Ker}( \left[\begin{array}{ll} \bar{A}^\top \,\, \bar{C}^\top \end{array}\right]^\top) \subseteq \operatorname{Ker}(\bar{A}_{cl})$. 

Combining both inclusions, we conclude that, $\operatorname{Ker}(\bar{A}_{cl})= \operatorname{Ker}( \left[\begin{array}{ll} \bar{A}^\top \,\, \bar{C}^\top \end{array}\right]^\top)$, which completes the proof.
\end{proof}
\begin{remark}
Actually, any nonzero vector in $\operatorname{Ker}( \left[\begin{array}{ll} \bar{A}^\top \,\, \bar{C}^\top \end{array}\!\right]\!^\top\!)$ is an eigenvector of $\bar{A}$ corresponding to the undetectable eigenvalue $0$ for the pair $(\bar{C},\bar{A})$. Thus, the steady state of the linear-quadratic optimal control problem \eqref{eq_optimal} is the eigenvector of $\bar{A}$ corresponding to the undetectable eigenvalue $0$.
\end{remark}
\subsection{Maximal steady-state space}
Consider the steady-state space defined in \eqref{steady_space}:
\begin{equation*}
\operatorname{Ker}\left(\begin{bmatrix}
\bar{A} \\
\bar{C}
\end{bmatrix}\right) = \operatorname{Ker}\left(\begin{bmatrix}
A & BH \\
F_{1} & F_{2} \\
C & 0
\end{bmatrix}\right).
\end{equation*}
In the multi-agent system described by \eqref{eq1}, the dynamics (including matrices \(A\) and \(B\)) are given. In contrast, the control strategy specified in \eqref{eq4} and \eqref{eq9} can be designed, involving the exogenous system (with matrices \(H\), \(F_1\), and \(F_2\)) and the performance index (matrix \(C\)). Regardless of the control design, the steady state must satisfy
\begin{equation}\label{x_ss}
Ax_{ss} = -BH\,w_{ss}.
\end{equation}
This condition confines the steady state to a particular subspace, which we define as the \emph{maximal steady-state space}. Discussing this space is crucial because it characterizes the theoretical limits of the formation patterns that can be achieved through subsequent control design. By understanding the maximal steady-state space, we can determine the full range of feasible formations and thereby guide the design process to realize as many formations as possible.

Without loss of generality, we assume that the system matrices in \eqref{eq1} have already been transformed into the Jordan normal form via a linear transformation, i.e.,
\begin{equation}\label{jordan1}
\begin{aligned}
A\!=\!\!\begin{bmatrix}
J(\lambda_{1}) & & & \\
& J(\lambda_{2})  & & \\
& & \ddots  & \\
& & &  J(\lambda_{l})
\end{bmatrix}\!,
B \!= \begin{bmatrix}
B_{\lambda_1} \\
B_{\lambda_2} \\
\vdots  \\
B_{\lambda_l}
\end{bmatrix},
\end{aligned}
\end{equation}
where $\lambda_{i} \neq \lambda_{j}, \forall i \neq j \in \mathcal{I}, ~\mathcal{I}=\{1,2,\cdots,l\}$; $J(\lambda_i)\in \mathbb{R}^{\sigma_{i} \times \sigma_{i}},~B_{\lambda_i}\in \mathbb{R}^{\sigma_{i} \times m}$, $\sigma_i$ is the algebraic multiplicity of $\lambda_{i}$. We group the Jordan blocks corresponding to the
nonzero eigenvalues into $A_r$ and explicitly write out each Jordan block corresponding to the eigenvalue $0$. Then equation \eqref{jordan1} can be rewritten as follows:
\begin{equation}\label{A_rB_r}
A =
\begin{bmatrix}
A_{r} &           &            &            &           \\
      & J_{1}(0)  &            &            &           \\
      &          & J_{2}(0)   &            &           \\
      &          &            & \ddots     &           \\
      &          &            &            & J_{\alpha_0}(0)
\end{bmatrix},
\quad
B =
\begin{bmatrix}
B_{r} \\[0.5pt]
B_{0}^{1} \\[1pt]
B_{0}^{2} \\[1pt]
\vdots   \\[1pt]
B_{0}^{\alpha_0}
\end{bmatrix},
\end{equation}
where $A_r \in \mathbb{R}^{r_0 \times r_0} $ is the block diagonal matrix composed of all Jordan blocks corresponding to the nonzero eigenvalues, $B_r \in \mathbb{R}^{r_0 \times m} $, $\alpha_0$ is the geometric multiplicity of the eigenvalue $0$ and 
\begin{equation}\label{A_rB_r2}
J_{j}(0)\!=\!
\begin{bmatrix}
0 & 1 & 0 & \cdots & 0\\
0 & 0 & 1 & \cdots & 0\\
\vdots & & \ddots & \ddots & \vdots\\
0 & 0 &        &  0       & 1\\
0 & 0 &        &          & 0
\end{bmatrix}\in \mathbb{R}^{r_{j} \times r_{j}},
B_{0}^{j}
\!\triangleq\!
\begin{bmatrix}
B_0^{j}(1:r_j-1)\\[2pt]
B_0^{j}(r_{j})
\end{bmatrix}\in \mathbb{R}^{r_{j} \times m}, \text{for } j \;=\;1,2,\dots,\alpha_0,
\end{equation}
where $\sum_{j=1}^{\alpha_{0}} r_{j}$ is the algebraic multiplicity of the eigenvalue $0$, $~B_0^{j}(1:r_j-1)\in \mathbb{R}^{(r_{j}-1) \times m}$ is the submatrix of $B_{0}^{j}$ formed by the first $(r_j-1)$ rows of $B_{0}^{j}$, $B_0^{j}(r_j)\in \mathbb{R}^{1 \times m}$ is the last row of $B_{0}^{j}$. 

With matrices $A$ and $B$ given as \eqref{A_rB_r}-\eqref{A_rB_r2}, we can derive the maximal steady-state space in the following theorem.
\begin{theorem}\label{theorem3}
(Maximal steady-state space). Consider the optimal control problem \eqref{eq_optimal}, where $A\in \mathbb{R}^{n\times n},~B\in \mathbb{R}^{n\times m}$ are given as \eqref{A_rB_r}-\eqref{A_rB_r2}. 
Under Assumptions \ref{assump2}-\ref{assump1}, for any linear exogenous system (namely for any matrices $F_{1}, F_{2}, G, H, K$), the steady state of the multi-agent system $x_{ss}$ must belong to
\begin{equation}\label{eq14}
\operatorname{Im}(\widetilde{X})=\operatorname{Im}(
\left[\begin{array}{{c:c}}
A_{r}^{-1}B_{r}U_{B_{\text{last}}} & \\
\widetilde{X}_1 & \\
\widetilde{X}_2 & \quad U_{A} \\  
\vdots & \\
\widetilde{X}_{\alpha_{0}} & 
\end{array}\right]),
\end{equation}
\begin{equation}\label{Blast}
\text{where } \widetilde{X}_i\!=\!\begin{bmatrix}
0 \\
B_{0}^{i}\left(1:r_{i}-1\right) U_{B_{\text{last}}}  \end{bmatrix},\text { for } i=1,2, \cdots, \alpha_{0} ,~
B_{\text{last}}=\begin{bmatrix}
B_0^{1\top}(r_1) &\,
B_0^{2\top}(r_2) &\,
\cdots &\,
B_0^{\alpha_0\top}(r_{\alpha_0})
\end{bmatrix}^\top,    
\end{equation}
$U_{B_{\text{last}}}$ and $U_{A}$ are the matrices whose columns form a standard orthonormal basis for $\operatorname{Ker}(B_{\text{last}})$ and $\operatorname{Ker}(A)$ respectively, $\widetilde{X}$ is a matrix whose columns are linearly independent.
\end{theorem}
\begin{proof} 
For any exogenous system, the steady state must satisfy $Ax_{ss}=-B Hw_{ss}$ by \eqref{x_ss}.
Since $H$ and $w_{ss}$ belong to the exogenous system, they can be arbitrarily chosen as long as they are finite. Define $\hat{w}=-Hw_{ss}\in \mathbb{R}^{m \times  1}$, then $\hat{w}$ can also be arbitrarily selected and the condition is equal to
\begin{equation}\label{eq_state_new}
Ax_{ss}=B \hat{w}.
\end{equation}
Define $x_{ss}= \begin{bmatrix}x_r^\top & x_1^\top & x_2 ^\top& \cdots & x_{\alpha_0}^\top\end{bmatrix}^\top$, where $x_r \in \mathbb{R}^{r_0},\ x_i \in \mathbb{R}^{r_i},\ \text{for }i=1,2,\ldots,\alpha_0.$
Then \eqref{eq_state_new} is equal to
\begin{equation}
A_{r} x_{r}=B_{r} \hat{w},~
J_{i}(0) x_{i}=B_{0}^{i} \hat{w},\text { for } i=1,2, \cdots, \alpha_{0} .
\end{equation}
Thus
\begin{equation}\label{x_r}
x_r = A_r^{-1}B_r\hat{w},~
x_i\!=\!\begin{bmatrix}\ast \\ B_0^i(1:r_i-1)\hat{w}\end{bmatrix},\text{for } i=1,2, \cdots, \alpha_{0},
\end{equation}
where $\ast\in \mathbb{R}^{1 \times m}$ is arbitrarily, $\hat{w}$ satisfies
$B_0^i(r_i)\hat{w} = 0,\text { for } i=1,2, \cdots, \alpha_{0} $.
\\Define $B_{\text{last}}=\begin{bmatrix}
B_0^{1\top}(r_1) &\,
B_0^{2\top}(r_2) &\,
\cdots &\,
B_0^{\alpha_0\top}(r_{\alpha_0})
\end{bmatrix}^\top$, then we have 
$B_{\text{last}}\hat{w}=0.$ Thus 
\begin{equation}\label{w}
\hat{w}\in \mathrm{Im}(U_{B_\text{last}}),
\end{equation}
where $U_{B_{\text{last}}}$ is the matrix whose columns form a standard orthonormal basis for $\operatorname{Ker}(B_{\text{last}})$.
In \eqref{x_r}, because $\ast\in \mathbb{R}^{1 \times m}$ is arbitrarily, we have
\begin{equation}\label{x_i_final}
x_i \in \mathrm{Im}(\left[\begin{array}{cc}
0 & 1 \\
B_0^i(1:r_i-1) U_{B_\text{last}}& 0_{(r_i-1)\times 1}
\end{array}\right]),  \text{for } i=1,2, \cdots, \alpha_{0}.
\end{equation}
Combining \eqref{x_r},\eqref{w},\eqref{x_i_final}, we have
\begin{equation}\label{eqx}
\left[\begin{array}{c}
x_{r} \\
x_{1} \\
x_{2} \\
\vdots \\
x_{\alpha_{0}}
\end{array}\right] \in \operatorname{Im}(
\left[\begin{array}{{c:c:c:c:c}}
A_{r}^{-1}B_{r}U_{B_{\text{last}}} & 0&0& &0\\
0 &1&0& &0 \\
B_{0}^{1}\left(1:r_{1}-1\right) U_{B_{\text{last}}} & 0&0& &0\\
0 & 0&1& &0  \\
B_{0}^{2}\left(1:r_{2}-1\right) U_{B_{\text{last}}} &0&0&\cdots &0  \\
\vdots &\vdots&\vdots& &\vdots \\
0 & 0&0& &1 \\B_{0}^{\alpha_{0}}\left(1:r_{\alpha_{0}-1}\right)U_{B_{\text{last}}} & 0&0& &0
\end{array}\right]),
\end{equation}
where the columns on the right-hand side (each containing a single ``1'' and zeros elsewhere) form exactly the Kernel space of \(A\). By collecting these vectors into a matrix \(U_A\), equation~\eqref{eqx} becomes equivalent to equation~\eqref{eq14}, which completes the proof.
\end{proof}
Define the space $\operatorname{Im}(\widetilde{X})$ as the maximal steady-state space, where the steady state $x_{ss}$ of the multi-agent system must belong to. The design of the exogenous system cannot force the state to converge to a formation outside $\operatorname{Im}(\widetilde{X})$; it can only determine which specific formation within $\operatorname{Im}(\widetilde{X})$ the state converges to. When \(A\) has no zero eigenvalues, \(\operatorname{Im}(\widetilde{X})=\operatorname{Im}(A^{-1}B)\). The following corollary considers the case in which all Jordan blocks associated with the zero eigenvalue of \(A\) are 1-dimensional. In this case, rewrite \eqref{A_rB_r} as 
\begin{equation}\label{A1-dimensional}
A =
\begin{bmatrix}
A_{r} &     0   \\   
 0     &    0_{(n-r_0)\times (n-r_0)}
\end{bmatrix},
\quad
B =
\begin{bmatrix}
B_{r} \\[0.5pt]
B_{0} \end{bmatrix},
\end{equation}
where $A_r \in \mathbb{R}^{r_0 \times r_0} $ is the block diagonal matrix composed of all Jordan blocks corresponding to the nonzero eigenvalues, $B_r \in \mathbb{R}^{r_0 \times m} $. 

\begin{corollary}\label{cor1}
Consider the optimal control problem \eqref{eq_optimal}, where $A\in \mathbb{R}^{n\times n},~B\in \mathbb{R}^{n\times m}$ are given as \eqref{A1-dimensional} and all Jordan blocks associated with the zero eigenvalue of $A$ are 1-dimensional. 
Under Assumptions \ref{assump2}-\ref{assump1}, for any linear exogenous system (namely for any matrices $F_{1}, F_{2}, G, H, K$), the steady state of the multi-agent system $x_{ss}$ must belong to
\begin{equation}
\operatorname{Im}(\widetilde{X})=\operatorname{Im}(
\left[\begin{array}{{c:c}}
A_{r}^{-1}B_{r}U_{B_{0}} & 0\\
0 & I_{n-r_0}
\end{array}\right]),
\end{equation}
where $U_{B_{0}}$ is the matrix whose columns form a standard orthonormal basis for $\operatorname{Ker}(B_{0})$, $\widetilde{X}$ is a matrix whose columns are linearly independent.
\end{corollary}

\section{Inverse problem: system design}
In Section 3, we established the existence condition for a nonzero steady state and discussed the steady-state space. However, in practice, our focus shifts to designing the system to ensure that the multi-agent system achieves the desired formation, thereby addressing the inverse problem. Designing the system requires three steps:

In step 1, for any desired formation, we construct a matrix \(B\) such that the corresponding maximal steady-state space encompasses the desired formation while ensuring the controllability of the pair \((A, B)\). Most importantly, our construction guarantees that \(B\) has the minimal rank, which directly minimizes the number of required control inputs.

In step 2, we design the exogenous system—comprising matrices \(H\), \(F_1\), \(F_2\), \(K\), and \(G\), along with the performance index matrix \(C\)—to ensure that the state converges to the desired formation under optimal control. Notably, our construction attains the minimal possible dimension for the exogenous system.

In step 3, we specify the initial state \(w(0)\) of the exogenous system to achieve the desired formation scaling.
\subsection{Step 1: design of the input matrix \(B\)}
Theorem \ref{theorem3} characterizes the maximal steady-state space \(\operatorname{Im}(\widetilde{X})\), which is fundamentally determined by the input matrix \(B\). In this step, we design \(B\) to meet two essential conditions: (1) the maximal steady-state space \(\operatorname{Im}(\widetilde{X})\) must encompass the desired formation, and (2) the pair \((A, B)\) must be controllable. Although the construction of \(B\) is not unique, our approach deliberately constructs \(B\) with the minimal rank, thereby minimizing the number of required control inputs.

Define a collection of $z$ desired formations as the space Im\((X^{df})\), where
\(X^{df}=[x_{df}^1 \quad x_{df}^2 \quad \cdots \quad x_{df}^z] \in \mathbb{R}^{n \times z}\) is a matrix, $x_{df}^i \in \mathbb{R}^{n} (i=1,2,\cdots,z)$ represents a desired formation vector. The reason for considering a collection
of desired formations rather than a single desired formation is that we can design one \(B\) and then
use different exogenous systems to achieve various formations within the collection, eliminating
the need to redesign \(B\) each time.
Because $\operatorname{Im}(\widetilde{X})$ must comprise $\operatorname{Im}(U_A)$ by \eqref{eq14}, in order to make 
$\operatorname{Im}(X^{df}) \subseteq \operatorname{Im}(\widetilde{X})$, 
we can remove the components of $\operatorname{Im}(X^{df})$ that belong to $\operatorname{Im}(U_A)$ and define the residual desired formation space as
\begin{equation}\label{res_desired}
\operatorname{Im}(X^{rf})=\operatorname{Im}(X^{df}) \cap(\operatorname{Im}(U_A))^{\perp},
\end{equation}
where  $(\operatorname{Im}(U_A))^\perp$ represents the orthogonal complement of $\operatorname{Im}(U_A)$, $X^{rf} \in \mathbb{R}^{n \times z_0}$ 
and its column vectors are linearly independent. Define \begin{equation}\label{X_rf}
X^{rf}=
\begin{bmatrix}
X_r^{rf\top}&
X_1^{rf\top}&
X_2^{rf\top}&
\cdots &
X_{\alpha_0}^{rf\top}
\end{bmatrix}^\top,    
\end{equation} 
where $X_r^{rf} \in \mathbb{R}^{r_0\times z_0},~X_i^{rf} \in \mathbb{R}^{r_i\times z_0}$.

In order to make $(A,B)$ controllable, we give a lemma for the controllability of the Jordan normal form system first. 
\begin{lemma} \label{lemma2}
(cf. \cite{chen1984linear}) For the Jordan normal form system \eqref{jordan1}, the necessary and sufficient condition for the complete controllability is
\begin{equation}\label{Jordan_control}
\operatorname{rank}(B_{\lambda_i})=\alpha_i, \forall i \in \mathcal{I},
\end{equation} 
where $\mathcal{I}=\{1,2,\cdots,l\}$, $\alpha_i$ is the geometric multiplicity of the eigenvalue $\lambda_{i}$ of $A$.
\end{lemma} 

Based on the above lemma, we propose a design for the matrix $B$ that ensures $\operatorname{Im}(X^{df}) \subseteq \operatorname{Im}(\widetilde{X})$ and that the pair $(A,B)$ is controllable. Although the design of $B$ is not unique, the matrix $B$ we designed achieves the minimal rank among all matrices satisfying these conditions.
\begin{theorem}\label{theorem4}(Design $B$).
Consider the multi-agent system \eqref{eq1}, where $A\in \mathbb{R}^{n \times  n}$ is given as $A$ in \eqref{A_rB_r}. For any desired formation space $\operatorname{Im}(X^{df})$, design the matrix \( B \) as
$B = \left[ B^{df} \quad  B^{c} \right]$,
\begin{equation}
B^{df} \triangleq
\begin{bmatrix}
Ar\,X_r^{rf}\\
B_0^1\\
B_0^2\\
\vdots\\
B_0^{\alpha_0}
\end{bmatrix},B_0^j \triangleq \begin{bmatrix}
X_j^{rf}\,(2: r_j)\\
0
\end{bmatrix} \text{for }j=1,2,\cdots,\alpha_0,
\end{equation}
where $X_j^{rf}$ is defined in \eqref{X_rf}, $X_j^{rf}(2:r_j)$ is formed by the last $(r_j-1)$ rows of $X_j^{rf}$. 
For each eigenvalue \(\lambda_i\) of \(A\) (with algebraic multiplicity \(\sigma_i\)), we define the submatrices \(B_{\lambda_i}^{df}\) and \(B_{\lambda_i}^{c}\) to have \(\sigma_i\) rows. These submatrices are obtained by partitioning the previously defined matrix \(B^{df}\) (and the corresponding matrix \(B^c\)) as follows:
\begin{equation}\label{Blambda}
\begin{aligned}
\begin{bmatrix}
B_{\lambda_1}^{df} & B_{\lambda_1}^{c} \\
B_{\lambda_2}^{df} & B_{\lambda_2}^{c} \\
\vdots & \vdots \\
B_{\lambda_l}^{df} & B_{\lambda_l}^{c}
\end{bmatrix}
\;=\;
\begin{bmatrix}
B^{df} & B^{c}
\end{bmatrix}.
\end{aligned}
\end{equation}
Then $B^c$ is chosen such that 
\begin{equation}
\operatorname{rank}([B_{\lambda_i}^{df}\,\, B_{\lambda_i}^{c} ])=\alpha_i,\text{for } i \in \mathcal{I},
\end{equation}
where \(\mathcal{I} = \{1,2,\ldots,l\}\), and \(\alpha_i\) is the geometric multiplicity of the eigenvalue \(\lambda_i\) of the matrix \(A\). Plus, the number of columns in \(B^c\) is given by $\max_{i \in \mathcal{I}}
\bigl\{
\alpha_i \;-\; \operatorname{rank}\bigl(B_{\lambda_i}^{df}\bigr)
\bigr\}$.

With such matrix $B$, there is $\operatorname{Im}(X^{df}) \subseteq \operatorname{Im}(\widetilde{X})$, where $\widetilde{X}$ is determined by \eqref{eq14} and $(A,B)$ is controllable. Plus the matrix $B$ achieves the minimal rank among all matrices satisfying these conditions.
\end{theorem}
\begin{proof}
In \eqref{eq14}, define matrix $\widetilde{X}_L$ such that $[\widetilde{X}_L\quad U_A]=\widetilde{X}.$
According to the definition of the residual desired formation space $\operatorname{Im}(X^{rf})$ in \eqref{res_desired}, in order to make $\operatorname{Im}(X^{df}) \subseteq \operatorname{Im}(\widetilde{X})$, we only need to make 
$\operatorname{Im}(X^{rf}) \subseteq \operatorname{Im}(\widetilde{X}_L)$.
Thus, we set $\widetilde{X}_L=X^{rf}$. Then by \eqref{eq14} and \eqref{X_rf}, we have
\begin{equation}
\widetilde{X}_L=\left[\begin{array}{{c}}
A_{r}^{-1}B_{r}U_{B_{\text{last}}}  \\
\widetilde{X}_1 \\
\widetilde{X}_2  \\  
\vdots  \\
\widetilde{X}_{\alpha_{0}}  
\end{array}\right]=X^{rf}= \begin{bmatrix}
X_r^{rf}\\
X_1^{rf}\\
X_2^{rf}\\
\vdots\\
X_{\alpha_0}^{rf}
\end{bmatrix},
\end{equation}
where 
\begin{equation}
\widetilde{X}_j\!=\!\begin{bmatrix}
0 \\
B_{0}^{j}\left(1:r_{j}-1\right) U_{B_{\text{last}}}  \end{bmatrix},\text { for } j=1,2, \cdots, \alpha_{0}.
\end{equation}
Without loss of generality, we can choose $B$ which satisfies
\begin{equation}\label{B_last}
B_{\text{last}} = 0,\,U_{B_{\text{last}}}=I,
\end{equation}
and 
\begin{equation}\label{B_r}
B_r =A_rX_r^{rf},    
\end{equation}
\begin{equation}\label{B_0i}
B_{0}^{j}\left(1: r_{j}-1\right)=X_{j}^{rf}\left(2:r_{j}\right).   
\end{equation}
From \eqref{B_last}, we have $B_{0}^{j}\left(r_{j}\right)=0$, for $j=1,2, \cdots \alpha_{0}$. 
And combining \eqref{B_r} and \eqref{B_0i}, we have
\begin{equation}\label{B^df}
B =\begin{bmatrix}
A_r\,X_r^{rf}\\
B_0^1\\
B_0^2\\
\vdots\\
B_0^{\alpha_0}
\end{bmatrix},B_0^j = \begin{bmatrix}
X_j^{rf}\,(2: r_j)\\
0
\end{bmatrix}, j=1,2, \cdots \alpha_{0}.
\end{equation}
Then $\operatorname{Im}(X^{df}) \subseteq \operatorname{Im}(\widetilde{X})$.

Then building on the condition that $\operatorname{Im}(X^{df}) \subseteq \operatorname{Im}(\widetilde{X})$, we further require that the pair $(A,B)$ is controllable. 
Define $B$ in \eqref{B^df} as $B^{df}$, which makes the desired formation $\operatorname{Im}(X^{df})$ belongs to $\operatorname{Im}(\widetilde{X})$.
According to Lemma \ref{lemma2}, for $(A,B^{df})$ to be controllable, it is necessary that $\operatorname{rank}(B^{df}_{\lambda_i})=\alpha_i, \forall i \in \mathcal{I}$, where $B^{df}_{\lambda_i}$ is defined in \eqref{Blambda}.
If $\exists \, i\in \mathcal{I}$ such that $ \alpha_i> \operatorname{rank}(B^{df}_{\lambda_i})$, then $(A,B^{df})$ is not controllable.
Thus we need to increase the rank of $B^{df}$ at least by $\max_{i\in \mathcal{I}} \{\alpha_i-\operatorname{rank}(B^{df}_{\lambda_i})\}$. This means the minimal rank of $B$ satisfying $\operatorname{Im}(X^{df}) \subseteq \operatorname{Im}(\widetilde{X})$ and $(A,B)$ controllable is 
\begin{equation}\label{eqrank}
\begin{aligned}
&\text{rank}(B^{df}) 
+ \max_{i \in \mathcal{I}} \{\alpha_i - \operatorname{rank}(B^{df}_{\lambda_i})\}.
\end{aligned}
\end{equation}
Let \(B^c\) be a matrix whose number of columns is $\max_{i\in \mathcal{I}} \{\alpha_i-\operatorname{rank}(B^{df}_{\lambda_i})\}$. Define $B=\begin{bmatrix}
B^{df} \quad  B^{c}
\end{bmatrix}$. We can always find \( B^{c}\) such that $\operatorname{rank}([B_{\lambda_i}^{df}\,\, B_{\lambda_i}^{c} ])=\alpha_i, \forall i\in \mathcal{I}$.
Hence $(A, B)$ is controllable. Since the rank of $B$ equals to \eqref{eqrank}, $B$ achieves the minimal rank among all matrices satisfying $\operatorname{Im}(X^{df}) \subseteq \operatorname{Im}(\widetilde{X})$ and that $(A, B)$ is controllable, which completes the proof.
\end{proof}
Then we consider some special cases for Theorem \ref{theorem4}, where $B^{df}$ in $B$ has a simpler expression.
\\(1) When the desired formation space $\operatorname{Im}(X^{df})$ belongs to $\operatorname{Im}(U_A)$, there is $X^{rf}=0$ and we have $B^{df}=0$.
\\(2) When \(A\) has no zero eigenvalues, we have \(B^{df}= AX^{df}\). 
\\(3) When all Jordan blocks associated with the zero eigenvalue of \(A\) are 1-dimensional, we have $ B^{df}=
\begin{bmatrix}
A_r X_{r}^{rf}\\0
\end{bmatrix}$.
\subsection{Step 2: design the exogenous system and the performance index}
For any desired formation space \(\operatorname{Im}(X^{df})\), an input matrix \(B\) can be designed with minimal rank such that \(\operatorname{Im}(X^{df}) \subseteq \operatorname{Im}(\widetilde{X})\) and the pair \((A, B)\) is controllable. Subsequently, to ensure that the system state converges to a specific formation \(\operatorname{Im}(x_{df})\) (with \(x_{df} \in \mathbb{R}^n\)) within the maximal steady-state space \(\operatorname{Im}(\widetilde{X})\), the exogenous system and the performance index must be appropriately designed.

It is important to note that the design of the exogenous system is not unique; the same formation may be achieved through different control strategies implemented by various exogenous systems, as illustrated in Case 3 of Section 5 (Illustrative examples). This flexibility prevents unauthorized parties from deducing the desired formation based on historical control inputs, thereby enhancing system security. Here, we propose a design that prioritizes minimizing the dimension of the exogenous system.

\begin{theorem}\label{theorem5}(Design exogenous system and performance index).
Consider the optimal control problem \eqref{eq_optimal}, where matrices $A\in \mathbb{R}^{n \times  n}$ and $B\in \mathbb{R}^{n \times  m}$ are given as \eqref{A_rB_r}. 
Consider a desired formation vector 
$x_{df}\neq0 \in \mathbb{R}^n $ such that $\operatorname{Im}(x_{df}) \in \operatorname{Im}(\widetilde{X})\setminus \{ V_{+}(A),V_{mg(0)\geq 2}(A)\}$,
where $\widetilde{X}$ is defined in \eqref{eq14}, $V_{+}(A)$ and $V_{\mathrm{mg}(0)\ge 2}(A)$ denote, respectively, the sets of eigenvectors of $A$ corresponding to eigenvalues with positive real parts and those associated with the zero eigenvalue having Jordan blocks of size at least 2. 
Then the control strategy described by \eqref{eq4} and \eqref{eq9} can drive the state of the system \eqref{eq1} to $\operatorname{Im}\,(x_{df})$ under the following exogenous system and the performance index matrix $C$:
\vspace{-5pt} 
\begin{subequations}
\begin{equation}\label{H_design}
H = U_{ B_{last}}\in \mathbb{R}^{m \times  k},
\end{equation}
\begin{equation}\label{F1F2}
[F_1~~ F_2] =[0_{k \times n} \; 0_{k \times k}],
\end{equation}
\begin{equation}\label{KG}
\{K\in \mathbb{R}^{m \times m}, G\in \mathbb{R}^{k \times  m} \| 
(\begin{bmatrix}
A & BH \\
F_1 & F_2 
\end{bmatrix}\!,\! 
\begin{bmatrix}
BK \\
G 
\end{bmatrix}) 
\text{ is stabilizable} \},
\end{equation}
\begin{equation}\label{C_design}
C = [U_{x_{df}^\top}]^\top \in \mathbb{R}^{(n-1) \times  n},
\end{equation}
\end{subequations}
where $B_{\text{last}}$ is defined in \eqref{Blast}, $U_{B_{\text{last}}}$ and $U_{x_{df}^\top}$ are matrices whose columns form a standard orthonormal basis for $\operatorname{Ker}(B_{\text{last}})$ and $\operatorname{Ker}(x_{df}^\top)$.
\end{theorem}
\begin{proof}
The proof is presented in three parts.

\textbf{Part 1:} This step shows the design 
of $H$ in \eqref{H_design} ensures that any formation in $\operatorname{Im}(\widetilde{X})$ is possible to be achievable.
The effect of $H$ is reflected in 
$B H w_{ss} = - A x_{ss}$.
Since $B$ is full column rank, we have
\begin{equation}\label{H_eq}
H w_{ss} = - (B^\top B)^{-1} B^\top A x_{ss}.
\end{equation}
If $H$ is chosen such that $\forall x_{ss} \in \operatorname{Im}(\widetilde{X})$, \eqref{H_eq} is satisfied, then $x_{ss}$ is possible to be achievable. Thus $H$ should satisfy
$- (B^\top B)^{-1} B^\top A\widetilde{X} \subseteq \operatorname{Im}(H)$.
When 
$\operatorname{Im}(H) = \operatorname{Im}((B^\top B)^{-1} B^\top A\widetilde{X})$,
$H$ achieves the minimal rank. Considering $AU_A=0$, we have
\begin{equation}
A \widetilde{X} =
\begin{bmatrix}
A_{r} &           &            &            &           \\
      & J_{1}(0)  &            &            &           \\
      &          & J_{2}(0)   &            &           \\
      &          &            & \ddots     &           \\
      &          &            &            & J_{\alpha_0}(0)
\end{bmatrix}
\begin{bmatrix}
A_r^{-1} B_r U_{B_{last}}\\
\widetilde{X}_1\\
\widetilde{X}_2\\
\vdots\\
\widetilde{X}_{\alpha_0}
\end{bmatrix}.
\end{equation}
Compute
\begin{equation}
J_{i}(0) \widetilde{X}_i = J_{i}(0)
\begin{bmatrix}
0 \\
B_{0}^{i}\left(1:r_{i}-1\right) U_{B_{\text{last}}}  \end{bmatrix}= \begin{bmatrix}
B_{0}^{i}\left(1:r_{i}-1\right) U_{B_{\text{last}}} \\0 \end{bmatrix}=B_{0}^{i}U_{B_{\text{last}}},
\end{equation}
for $i=1,2,\dots,\alpha_0$, where we use $B_{0}^{\,i}(r_i) U_{B_{last}} = 0.$
Then we have
$A \widetilde{X} = B \,U_{B_{last}}$.
Finally, we have
$\mathrm{Im}(H) = \operatorname{Im}((B^\top B)^{-1} B^\top A\widetilde{X})=
\mathrm{Im}\bigl((B^\top B)^{-1} B^\top B \,U_{B_{last}}\bigr)
= \mathrm{Im}(U_{B_{last}})$.
Since \(U_{B_{last}}\) is the matrix whose columns form a standard orthonormal basis for \(\operatorname{Ker}(B_{last})\), choosing \(H = U_{B_{last}}\) minimizes the number of columns in \(H\). As the dimension of the exogenous system is equal to the number of columns in \(H\), this choice makes the exogenous system have the minimal dimension.

\textbf{Part 2:}
In this step, we demonstrate that our design of $F_1$, $F_2$, and $C$ in \eqref{F1F2} and \eqref{C_design} renders $\begin{bmatrix}x_{df} \\ w_0\end{bmatrix}$ the unique (up to a scalar multiple) nonzero vector in the steady-state space $\operatorname{Ker}(\bar{A}_{cl})$, where $w_0$ satisfies $Ax_{df} + BHw_0 = 0$.
By Theorem \ref{theorem0}, $\operatorname{Ker}(\bar{A}_{cl}) = \operatorname{Ker}( \left[\begin{array}{ll}
\bar{A}^\top \,\, \bar{C}^\top
\end{array}\right]^\top)$. Therefore, it suffices to prove that $\operatorname{Ker}( \left[\begin{array}{ll}
\bar{A}^\top \,\, \bar{C}^\top
\end{array}\right]^\top)=\operatorname{Im}\,(\begin{bmatrix}
x_{df} \\
w_0
\end{bmatrix})$. 
Define $\begin{bmatrix}
x \\
w
\end{bmatrix} 
\in \operatorname{Ker}( \left[\begin{array}{ll}
\bar{A}^\top \,\, \bar{C}^\top
\end{array}\right]^\top)$, $x \in \mathbb{R}^n$, $w \in \mathbb{R}^k$. When \([F_1~ F_2] = [0 \; 0]\), $C=[U_{x_{df}^\top}]^\top$, we have
\begin{subequations}\label{eq:line}
\begin{equation}\label{eq:line1}
A x + B Hw = 0,
\end{equation}
\begin{equation}\label{eq:line2}
[U_{x_{df}^\top}]^\top x = 0 .
\end{equation}
\end{subequations}
Because $U_{x_{df}^\top}$ is a matrix whose columns form a standard orthonormal basis for $\operatorname{Ker}(x_{df}^\top)$, the unique (up to a scalar multiple) nonzero solution of \eqref{eq:line2} is $x_{df}$. Thus if $w_0$ satisfies $A x_{df} + B H w_0 = 0$,  
$\begin{bmatrix}
x_{df} \\
w_0
\end{bmatrix}$ is one solution for \eqref{eq:line}.  
Another possible solution is when $x= 0, \exists w \neq 0$ which satisfies $B Hw=0$. But such $w$ does not exist because $B$ and $H$ are full column rank matrices.
Thus $\operatorname{Ker}( \left[\begin{array}{ll}
\bar{A}^\top \,\, \bar{C}^\top
\end{array}\right]^\top)=\operatorname{Im}\,(\begin{bmatrix}
x_{df} \\
w_0
\end{bmatrix})$.

\textbf{Part 3:} 
This step shows that Assumptions \ref{assump2} and \ref{assump1} are satisfied. Because $K$ and $G$ in \eqref{KG} ensure that the system is stabilizable, Assumption \ref{assump2} is satisfied. 
Then we prove Assumptions \ref{assump1}. If there exists an undetectable eigenvalue $\lambda$ for $\bar{A}$, then $\exists \phi \neq0$ which satisfies 
$\bar{A}\phi=\lambda\phi,~  \bar{C}\phi=0$.
Because $\bar{A} = \begin{bmatrix}
    A & BH\\
    0 & 0
\end{bmatrix}, \,\, \bar{C} = \begin{bmatrix}
    C & 0
\end{bmatrix}$, define $\phi = \begin{bmatrix} \phi_1\\ \phi_2 \end{bmatrix}$, then we have
\begin{equation}\label{con1}
A \phi_1 + B H \phi_2 = \lambda \phi_1,
\end{equation}
\begin{equation}\label{con2}
\lambda \phi_2 = 0,
\end{equation}
\begin{equation}\label{con3}
C\phi_1 = 0.
\end{equation}
Because $C =[U_{x_{df}^\top}]^\top$ in \eqref{con3}, we have 
\begin{equation}\label{con4}
\phi_{1} =x_{df} \text{ or } 0.
\end{equation}
When $\lambda \neq 0$, because of \eqref{con2}, $\phi_2 = 0$. Because $\phi \neq 0$, $\phi_1$ can only be $x_{df}$ to satisfy \eqref{con4}. Then \eqref{con1} becomes $A x_{df}= \lambda x_{df}$. 
Because $x_{df}$ can not be chosen from $V_{+}(A)$, $\mathrm{Re}(\lambda) \leq 0$. Thus $
\bar{A}$ has no undetectable eigenvalue with positive real part.

When $\lambda =0$, we have:
\begin{equation}\label{con5}
A \phi_1 + B H \phi_2 =0,
\end{equation}
\begin{equation}
C\phi_1 = 0.
\end{equation}
Thus $\phi_1$ also satisfies \eqref{con4}. If $\phi_1 = 0$, because $B H$ is column full rank in \eqref{con5}, we have $\phi_2 = 0$, which contradicts $\phi \neq 0$. Thus we have $\phi_1 =x_{df}$.   
Because $x_{df} \in \operatorname{Im}(\widetilde{X})$, there must exist $\phi_2$ satisfying $A x_{df} + B H\phi_2 = 0$. Define $\phi_2=w_0$, then $\bar{A}$ must have undetectable eigenvalue $0$ with eigenvector
$\begin{bmatrix}
    x_{df}\\ w_0
\end{bmatrix}$. 
Then we need to prove the algebraic multiplicity and geometric multiplicity are equal for the undetectable eigenvalue $0$. If the condition is not satisfied, $\exists \hat{\phi} \neq 0$ such that $\bar{A}\hat{\phi}=\phi$ and $\bar{C} \hat{\phi}= 0$, where $\phi$ satisfies $\bar{A}\phi=0$ and $\bar{C}\phi= 0$. Because $\operatorname{Ker}( \left[\begin{array}{ll}
\bar{A}^\top \,\, \bar{C}^\top
\end{array}\right]^\top)=\operatorname{Im}\,(\begin{bmatrix}
x_{df} \\
w_0
\end{bmatrix})$ as proved in \textbf{Part 2},
$\phi$ can only be $\begin{bmatrix}
    x_{df}\\ w_0
\end{bmatrix}$, where $w_0$ satisfies 
\begin{equation}\label{con_w0}
A x_{df} + B Hw_0 = 0.
\end{equation}
Define $\hat{\phi}=\begin{bmatrix}\hat{\phi}_{1}\\ \hat{\phi}_{2}
\end{bmatrix}$, then we have
\begin{equation}\label{ncon1}
A \hat{\phi}_{1} + BH \hat{\phi}_{2} = x_{df},
\end{equation}
\begin{equation}\label{ncon2}
w_{0} = 0,
\end{equation}
\begin{equation}\label{ncon3}
C\hat{\phi}_{1} = 0.
\end{equation}
Because $C =[U_{x_{df}^\top}]^\top$ in \eqref{ncon3},  we have 
\begin{equation}\label{phi12}
\hat{\phi}_{1} =x_{df} \text{ or } 0.
\end{equation}
Because of \eqref{con_w0} and \eqref{ncon2}, we have 
\begin{equation}\label{x_df}
A x_{df}= 0.
\end{equation}
Then we take \eqref{phi12} into \eqref{ncon1} and get 
\begin{equation}\label{con_z2}
B H \hat{\phi}_2
=
x_{df}.
\end{equation}
Because $H = U_{B_{last}}$, we can prove that $\mathrm{Im}(B H) \subseteq \mathrm{Im}(A)$. 
Then $\exists\, \varphi \in \mathbb{R}^{n}$ which satisfies $B H\hat{\phi}_2 = A\varphi$. Considering \eqref{con_z2}, $A\varphi=x_{df}$. Because $x_{df}$ is the eigenvector of 0 for $A$ by \eqref{x_df}, $\varphi$ is a generalized-eigenvector of $0$ for $A$. Because $x_{df}$ cannot be chosen from $V_{\mathrm{mg}(0)\ge 2}(A)$, such $\varphi$ does not exist. Thus, there does not exist $\hat{\phi}_2$, which indicates that $\bar{A}$ does not have any generalized-eigenvector of 0 corresponding to undetectable subspace. Therefore, Assumption \ref{assump1} is satisfied. 

Based on the proofs in \textbf{Parts 1–3}, we conclude that $\bar{A}_{cl}$ is marginally stable and
the state of the closed-loop system \eqref{closeloop} converge to $\operatorname{Im}\,(\begin{bmatrix}
x_{df} \\
w_0
\end{bmatrix})$. 
Consequently, the state of system \eqref{eq1} will converge to $\operatorname{Im}\,(x_{df})$. This completes the proof.
\end{proof}
\begin{remark}
In Theorem \ref{theorem5}, we set \(H = U_{B_{last}}\). It is possible when \(\operatorname{Ker}(B_{last}) = \{0\}\). In this case, \(\operatorname{Im}(\widetilde{X}) = \operatorname{Im}(U_A) = \operatorname{Ker}(A)\), which means the steady state is driven at most towards the null space of the matrix $A$. Actually, such formation can be achieved by the original optimal control problem~\eqref{origin_op}, which does not involve an exogenous system. Actually, Theorem \ref{theorem5} is compatible with this case, because the exogenous system design in Theorem \ref{theorem5} is \(H = F_1 = F_2 = G = 0 \in \mathbb{R}\), \(\{K \in \mathbb{R}^{m \times m} \mid (A, BK) \text{ is stabilizable}\}\), and \(C = [U_{x_{df}^\top}]^\top\). If \((A, B)\) is controllable, we may choose \(K = I_m\). With this design, the optimal control problem with the exogenous system~\eqref{eq_optimal} reduces to the original optimal control problem~\eqref{origin_op}.
\end{remark}

\subsection{Step 3: design the initial state of the exogenous system}
Based on Lemma~\ref{theorem1}, the steady state of the optimal control problem \eqref{eq_optimal} is related to the initial state of the combined system, \(\bar{x}(0)\). Typically, the initial state of the multi-agent system, \(x(0)\), is fixed, whereas the initial state of the exogenous system, \(w(0)\), can be any bounded value.
Hence, by selecting an appropriate \(w(0)\), one can achieve the desired formation scaling. The following theorem describes the method for determining $w(0)$.

\begin{algorithm}[h]
    \caption{System design to achieve a desired formation $d\cdot x_{df}$}
\begin{algorithmic}
    \STATE {\textbf{Input} System matrix $A$, collection of desired formations $ \operatorname{Im}\, (X^{df})$, desired formation with scaling $d\cdot x_{df}$, initial state $x(0)$ and input matrix $B$ (may be omitted).}
    \\ \textbf{Step 1: Construct $B$ if not provided.}
    \STATE {$B\leftarrow$ getInputMatrixB()}
    \IF{$B$ is not provided}
    \STATE { Compute $\operatorname{Im}(X^{rf})=\operatorname{Im}(X^{df})\cap(\operatorname{Im}(U_A))^{\perp}$;}
     \STATE {Compute $B^{df}$ and $B^{c}$ according to Theorem \ref{theorem4};}
    \STATE {Construct $B=\left[  B^{df} \quad  B^{c} \right].$}
      \RETURN $B$
        \ELSE 
        \STATE {Skip constructing $B$.}
\ENDIF
\\ \textbf{Step 2: Construct the exogenous system and the performance index.}
\STATE {Compute $\widetilde{X}$ according to \eqref{eq14}.}
    \IF {$\operatorname{Im} (x_{df}) \subseteq \operatorname{Im} (\widetilde{X})\setminus \{ V_{+}(A),V_{mg(0)\geq 2}(A)\}$}
    \STATE {Construct $H= U_{B_{last}}$}.
    \STATE {Construct matrix $[F_1 \; F_2]=[0_{k \times n}\,\, 0_{k \times k}]$}.
       \STATE {Construct matrices $K$ and $G$ such that $(\bar{A}, \bar{B})$ is stabilizable}.
    \STATE {Construct matrix $C =[U_{x_{df}^\top}]^\top.$}
\RETURN $H, F_1, F_2, K, G$ and $C$.
    \ELSE 
   \STATE{Terminate algorithm with message: \textit{The formation $d \cdot x_{df}$ cannot be achieved.}}
    \ENDIF
\\ \textbf{Step 3: Construct the initial state $w(0)$.}
\STATE {Compute the closed-loop matrix $\bar{A}_{cl}$ and $\phi_{x},\phi_{w}$}.
    \IF{$\phi_{w}\neq 0$ or $d\frac{\left\|x_{df}\right\|}{\|\phi_x\|}=\phi_{x}^{\top}x(0)$}
\STATE $w(0)$ is selected satisfying: $\phi_{w}^{\top} w(0)
=d\frac{\left\|x_{df}\right\|}{\|\phi_x\|}-\phi_{x}^{\top}x(0)$.
\RETURN $w(0)$.
    \ELSE
   \STATE {Terminate algorithm with message: \textit{The formation $d\cdot x_{df}$ cannot be achieved.}}
\ENDIF
\end{algorithmic} \label{alg:1}
\end{algorithm} 

\begin{theorem}\label{scaling}(Design the initial state of the exogenous system).
Suppose the exogenous system and the performance index are designed such that the state of system \eqref{eq1} will converge to $\operatorname{Im}\,(x_{df})$ under the control strategy described by \eqref{eq4} and \eqref{eq9}, where 
$x_{df} \in  \mathbb{R}^{n}$. 
Then the closed-loop matrix $\bar{A}_{\mathrm{cl}}$ has only one eigenvector $\phi \in \mathbb{R}^{(n+k) \times 1}$ of the zero eigenvalue, which satisfies $\bar{A}_{cl}\phi=0$ and $\phi^{\top}\phi=1$.
Denote $\phi = 
\begin{bmatrix}
\phi_{x}\\[6pt]
\phi_{w}
\end{bmatrix}$,
where $\phi_{x} \in \mathbb{R}^{ n},\;
\phi_{w}\in \mathbb{R}^{k}$.
Given the initial state of the multi-agent system \(x(0)\) and a desired formation scaling $d\cdot x_{df}$ with $d \neq 0$, 
this scaling is achieved when the initial state of the exogenous system $w(0)$ satisfies
\begin{equation}\label{w_0}
\phi_{w}^{\top} w(0)
=d\frac{\left\|x_{df}\right\|}{\|\phi_x\|}-\phi_{x}^{\top}x(0),
\end{equation}
where $\left\|x_{df}\right\|$ and $\|\phi_x\|$ represent the norms of the vectors $x_{df}$ and $\phi_x$ respectively.
\end{theorem}
\begin{proof}
Because the closed-loop matrix $\bar{A}_{\mathrm{cl}}$ has only one eigenvector $\phi$ of the zero eigenvalue, according to Lemma~\ref{theorem1}, 
\begin{align}
\bar{x}_{ss}= (\phi^\top \bar{x}(0)) \phi=(\phi_{x}^{\top}x(0)+\phi_{w}^{\top}w(0))\phi.
\end{align}
Considering $\bar{x}_{ss}=\begin{bmatrix} x_{ss} \\ w_{ss} \end{bmatrix}$ and $\phi = 
\begin{bmatrix}
\phi_{x}\\[6pt]
\phi_{w}
\end{bmatrix}$, we have 
\begin{align}
x_{ss}=(\phi_{x}^{\top}x(0)+\phi_{w}^{\top}w(0))\phi^x.
\end{align}
Because $w(0)$ satisfies \eqref{w_0}, there is
\begin{equation}
x_{ss}=d\frac{\left\|x_{df}\right\|}{\|\phi^x\|}\cdot \phi^x=d\cdot x_{df},
\end{equation}
the above equation holds because \(x_{df}\) and \(\phi^x\) are linearly dependent. Therefore, the scaling is achieved and the proof is complete.
\end{proof}

Algorithm 1 summarizes the design of system to achieve a desired formation $d\cdot x_{df}$, which includes three steps. In the first step, the algorithm checks whether an input matrix $B$ is provided. If not, it then constructs $B$ according to Theorem~\ref{theorem4}, which lets $\operatorname{Im}(X^{df}) \subseteq \operatorname{Im}(\widetilde{X})$ and $(A,B)$ controllable. If $B$ is provided, the algorithm skips constructing $B$. 
In the second step, it construct the exogenous system and
performance index according to Theorem \ref{theorem5} if $\operatorname{Im}(x_{df}) \in \operatorname{Im}(\widetilde{X})\setminus \{ V_{+}(A),V_{mg(0)\geq 2}(A)\}$. Otherwise, it terminates, stating that \textit{“The formation $d\cdot x_{df}$ cannot be achieved”}. 
In the third step, the algorithm checks if $\phi_w \neq 0$ or $d\frac{\left\|x_{df}\right\|}{\|\phi_x\|}=\phi_{x}^{\top}x(0)$. If true, it computes the initial state $w(0)$ according to Theorem \ref{scaling}. Otherwise, the algorithm terminates with the message: \textit{“The formation $d\cdot x_{df}$ cannot be achieved.”}
\section{Illustrative examples}
In this section, we will illustrate the effectiveness of our main results. Consider the state evolution equation of four agents in three-dimensional space
\begin{equation}\label{example1}
\dot{\hat{x}}=\hat{A}\hat{x}+\hat{B} u=(-L \otimes I_{3})\hat{x}+\hat{B} u,
\end{equation}
where $L = \begin{bmatrix}
2  & -1 & 0  & -1 \\
-1 & 3  & -2 & 0  \\
0  & -2 & 5  & -3 \\
-1 & 0  & -3 & 4
\end{bmatrix}$. We introduce a linear transformation to bring $\hat{A}$ into its Jordan normal form $A$, 
$A=\operatorname{diag}(0,-2.4746,-3.3691,-8.1563)\otimes I_{3}$,
whose all Jordan blocks associated with the zero eigenvalue are 1-dimensional. 
Under the Jordan normal form, we use theories presented in the aforementioned paper to compute required matrices and states. 
Then we apply an inverse linear transformation to obtain the corresponding matrices and states for system \eqref{example1}. It is worth noting that the exogenous system (including the state $w$ and the matrices $H, F_1, F_2, K,G$) remain unchanged before and after the linear transformation. This can be interpreted as the exogenous system undergoing an identity transformation.
For simplicity, we will directly present the results for system \eqref{example1}.

\textbf{Case 1} is an example where only parts of the agents are controlled. Consider $\hat{B}\!=\left[\begin{array}{c}
I_{9} \\
0_{3 \times 9}
\end{array}\right]$, where only agents 1-3 are controlled directly. We compute the maximum steady-state space $\operatorname{Im} (\widetilde{X})$ according to equation \eqref{eq14} and select a desired formation within $\operatorname{Im} (\widetilde{X})$. The desired formation for system \eqref{example1} is
\begin{equation*}
x_{df}^1=\begin{bmatrix}
1 & 1 & 0 & 10/3 & 10/3 & 0 & 7/3 & 7/3 & 0 & 2 ~ 2 ~ 0\end{bmatrix}^\top,
\end{equation*}
which corresponds to a line formation. The initial state of system \eqref{example1} is selected arbitrarily as $\hat{x}(0)=\left[\begin{array}{llllllllllll} 1.5&1&0.5& 3.3&3&0.3&3&2&0.2&2&1.5&0\end{array}\right]^{\top}$. Then we use Algorithm 1 to design the exogenous system, the performance index matrix $C$ and the initial state of exogenous system, where $w(0)\approx \left[\begin{array}{ll}0.2684 & 0\end{array}\right]^{\top}.$ Under these conditions, the evolution of the states and the trajectories of agents are shown in Fig. \ref{fig:1-1} and Fig. \ref{fig:1-2} respectively. As shown in Fig. \ref{fig:1-1}, the states of the agents tend to stabilize, and for sufficiently large \(t\) we have $\hat{x}\approx \left[\begin{array}{llllllllllll}1 & 1 & 0 & 3.3 & 3.3 & 0 & 2.3 & 2.3 & 0 & 2 & 2 & 0
\end{array}\right]^{\top}\approx x_{df}^1$. Fig. \ref{fig:1-2} shows the trajectories of agents, where the agents ultimately form a line formation. 
\begin{figure}[h]
  \centering
  \begin{minipage}{0.5\linewidth}
    \centering
    \includegraphics[width=\linewidth]{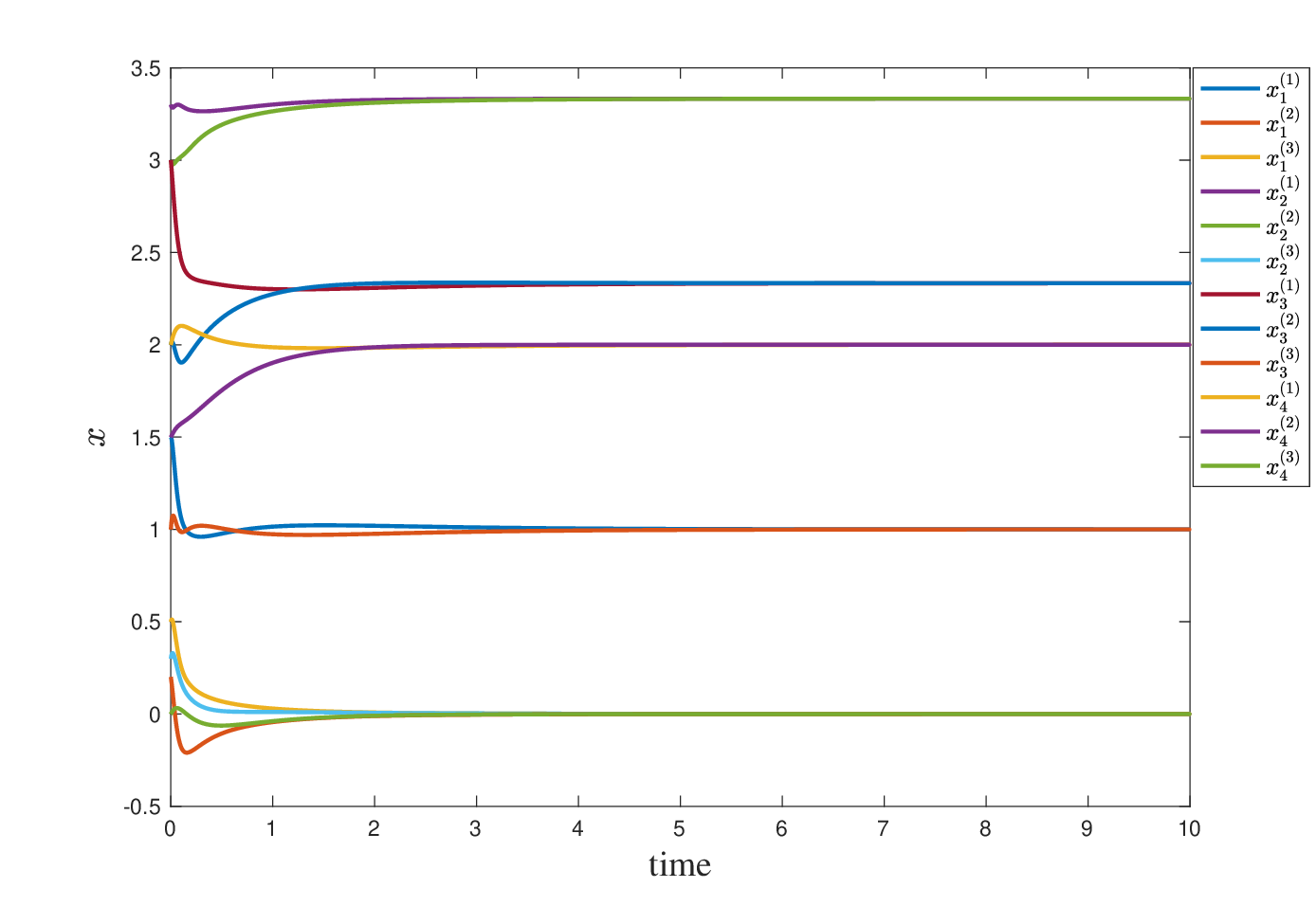}
    \caption{Evolution of states for line formation $x_{df}^1$}
    \label{fig:1-1}
  \end{minipage}\hfill
  \begin{minipage}{0.5\linewidth}
    \centering
    \includegraphics[width=\linewidth]{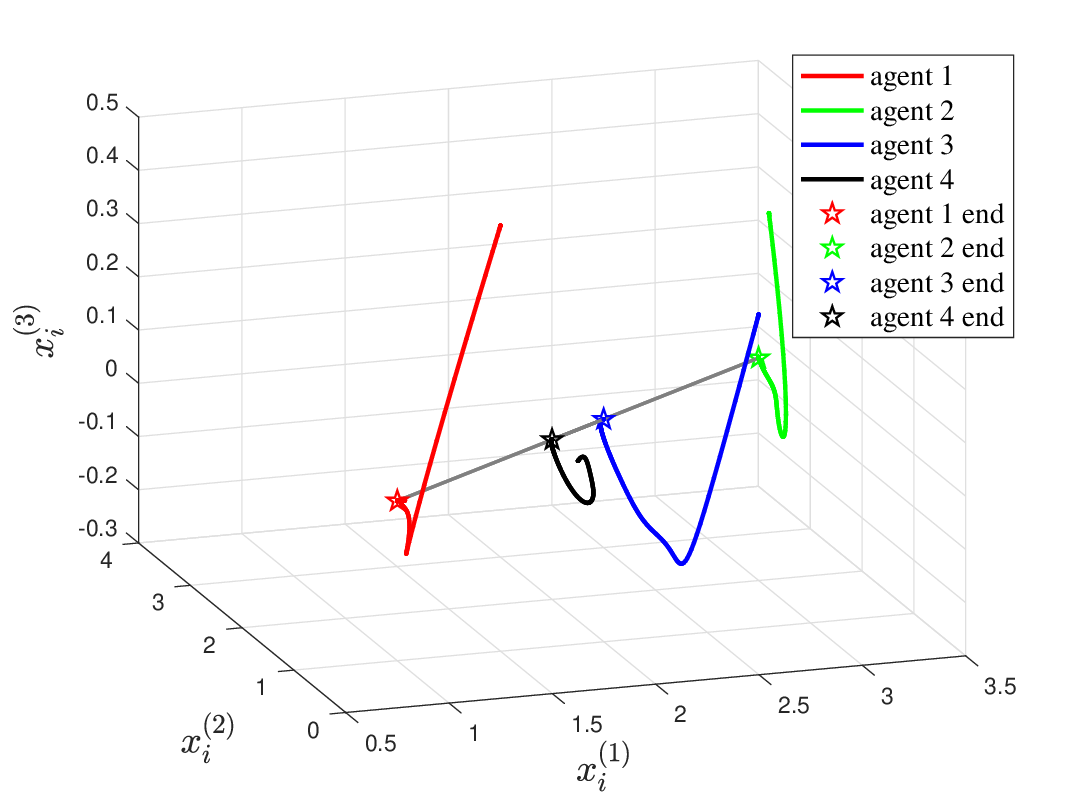}
    \caption{Trajectories of agents for line formation $x_{df}^1$}
    \label{fig:1-2}
  \end{minipage}
  \label{fig:case1}
\end{figure}

\begin{figure}[h]
  \centering
  \begin{minipage}[b]{0.49\linewidth}
    \centering
    \includegraphics[width=\linewidth]{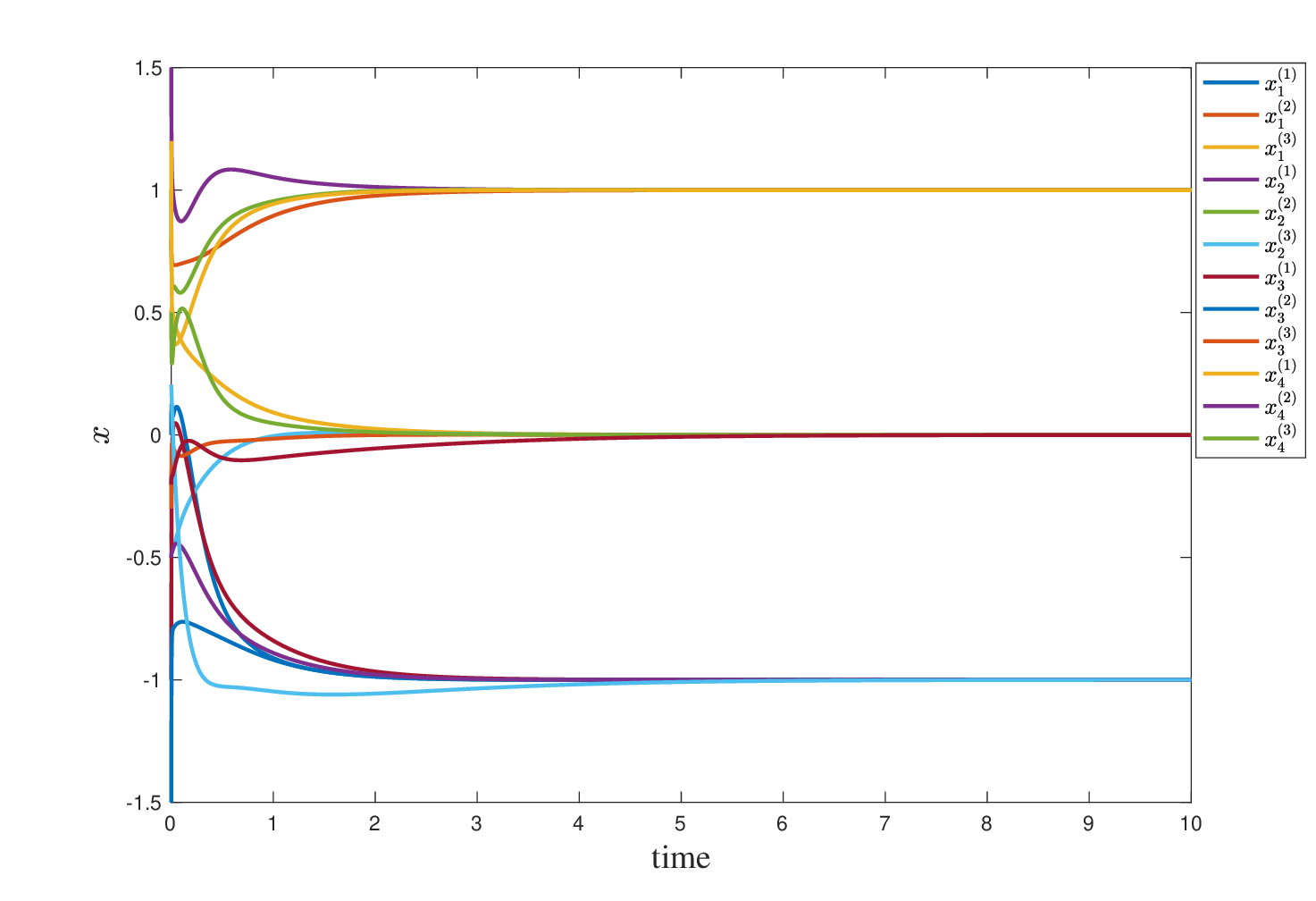}
    \captionof{figure}{Evolution of states for square formation $x_{df}^2$}
    \label{fig:2-1}
  \end{minipage}\hfill
  \begin{minipage}[b]{0.49\linewidth}
    \centering
    \includegraphics[width=\linewidth]{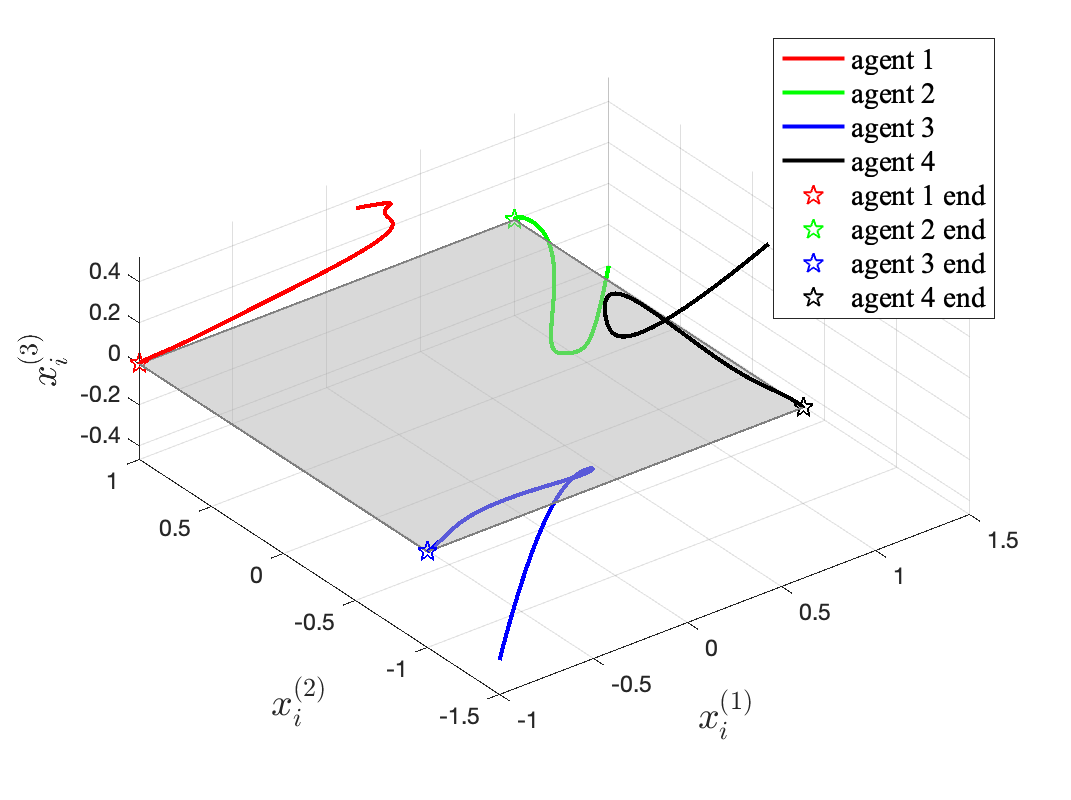}
    \captionof{figure}{Trajectories of agents for square formation $x_{df}^2$}
    \label{fig:2-2}
  \end{minipage}
  \begin{minipage}[b]{0.49\linewidth}
    \centering
    \includegraphics[width=\linewidth]{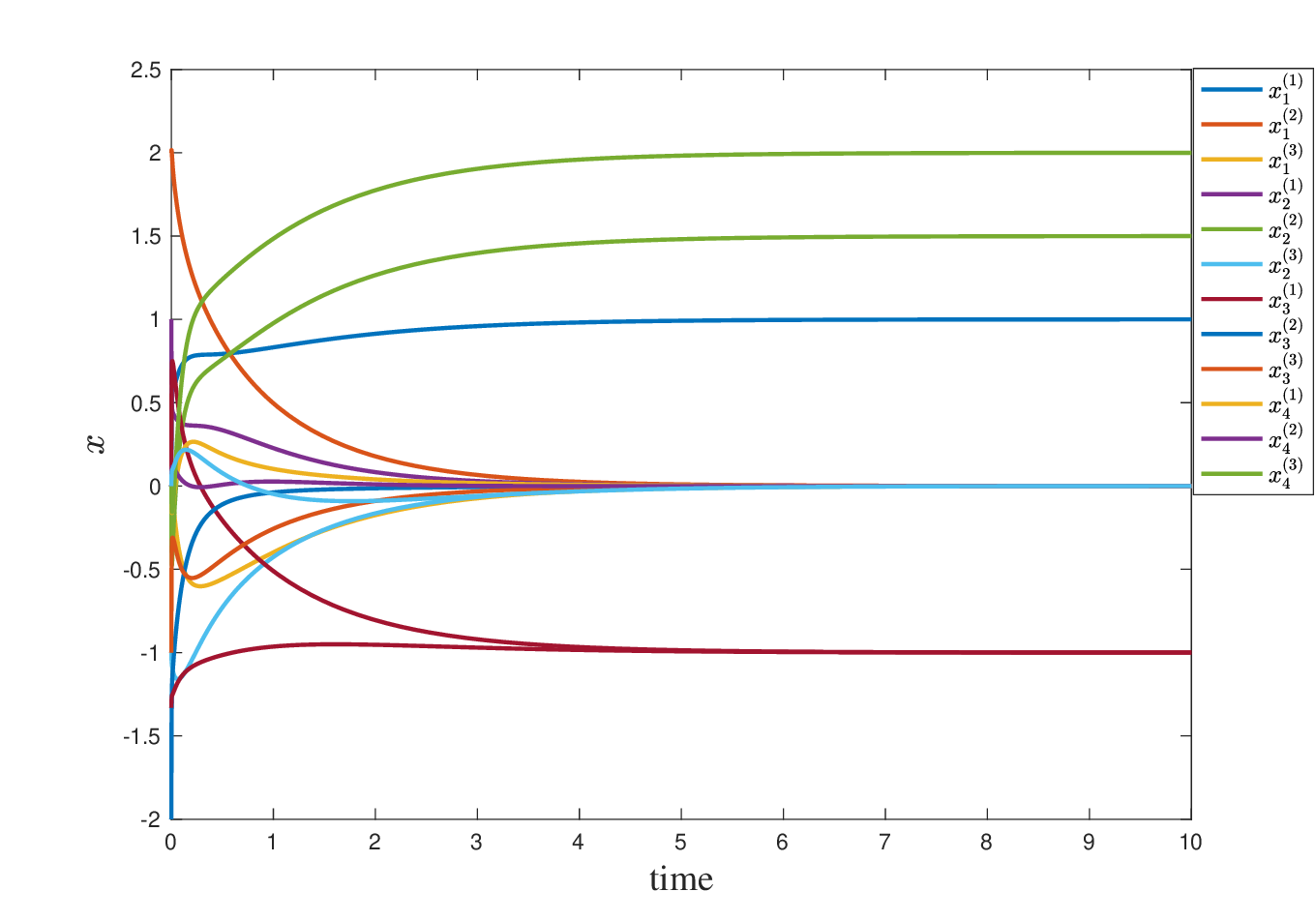}
    \captionof{figure}{Evolution of states for tetrahedron formation $x_{df}^3$}
    \label{fig:2-3}
  \end{minipage}\hfill
  \begin{minipage}[b]{0.49\linewidth}
    \centering
    \includegraphics[width=\linewidth]{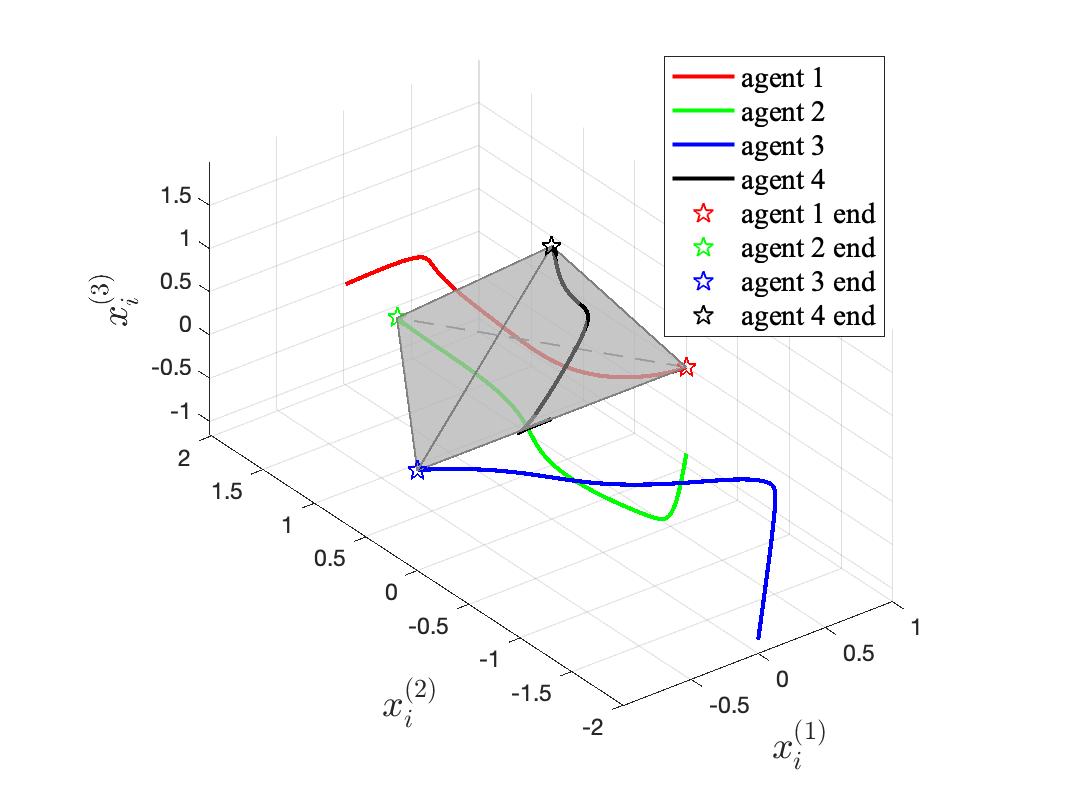}
    \captionof{figure}{Trajectories of agents for tetrahedron formation $x_{df}^3$}
    \label{fig:2-4}
  \end{minipage}
  \label{fig:case2}
\end{figure}

\textbf{Case 2} is an example where the input matrix $\hat{B}$ is not provided.
Consider a collection of two desired formations which includes the square formation $x_{df}^2$ and the regular tetrahedron formation $x_{df}^3$, where
\begin{equation*}
\begin{aligned}
x_{df}^2&=\begin{bmatrix}
-1 & 1 & 0 & 1 & 1 & 0 & -1 & -1 & 0 & 1 ~ -1 ~ 0\end{bmatrix}^\top,\\
x_{df}^3&= \begin{bmatrix}
1 & 0 & 0 & 0 & 1.5 & 0 & -1 & 0 & 0 & 0 ~ 0 ~ 2
\end{bmatrix}^\top.
\end{aligned}
\end{equation*}
We define $X^{df}= \left[x_{df}^2, x_{df}^3\right]$ and compute $\hat{B}$ according to Algorithm 1 which makes $\operatorname{Im}(X^{df})$ belong to the maximum steady-state space. Then we will design different exogenous systems to achieve $x_{df}^2$ and $x_{df}^3$ respectively. 

We firstly consider $x_{df}^2$ and choose the initial state arbitrarily as 
\begin{equation*}
\hat{x}(0)\!=\!\left[\begin{array}{llllllllllll}0&0.8&0.5&1.5&1&-0.4&-1&-1.5&-0.3&1.2&-0.5&0.5\end{array}\!\right]\!^{\top}.
\end{equation*}
Then we use Algorithm 1 to design the exogenous system, determine the performance index matrix $C$ and get the initial state of exogenous system $w(0)\approx \left[\begin{array}{llllllllllll}0.0345 & 0\end{array}\right]^{\top}.$ Fig. \ref{fig:2-1} shows that the states of the agents tend to stabilize; for sufficiently large \(t\), we have $\hat{x}\approx\left[\begin{array}{llllllllllll}-1 & 1 & 0 & 1 & 1 & 0 & -1 & -1 & 0 & 1 ~ -1 ~ 0\end{array}\!\right]^{\top}\!=x_{df}^2$. This result is consistent with the trajectories depicted in Figure~\ref{fig:2-2}, where the agents eventually form a square formation.

Similarly, we design another exogenous system to achieve the tetrahedron formation $x_{df}^3$. The initial state is set arbitrarily as $\hat{x}(0)\!=\!\left[\begin{array}{llllllllllll} 0&2&0& 1&0&-1&0&-2&-1&0&0&0\end{array}\right]^{\top}.$ Applying Algorithm~1 yields $w(0)\approx \left[\begin{array}{llllllllllll}-2.3559 & 0\end{array}\right]^{\top}\!$. Figure~\ref{fig:2-3} shows that for sufficiently large \(t\), we obtain
$\hat{x}\approx\left[\begin{array}{llllllllllll}1 & 0 & 0 & 0 & 1.5 & 0 & -1 & 0 & 0 & 0 ~ 0 ~ 2\end{array}\right]^{\top}$ $=x_{df}^3$. This is corroborated by the trajectories in Figure~\ref{fig:2-4}, where the agents ultimately form a tetrahedron formation. 

Moreover, in Figures \ref{fig:1-2}, \ref{fig:2-2}, and \ref{fig:2-4}, the agents do not travel from their initial positions to the desired formation along the shortest path. Instead, they follow more circuitous routes, which helps to confuse potential adversaries who might attempt to intercept or seize the formation. This arises because, in the optimal control problem \eqref{eq_optimal}, the objective is to minimize the input of the exogenous system rather than that of the multi-agent system.

\begin{figure}[htbp]
  \centering
  \begin{minipage}{0.49\linewidth}
    \centering
    \includegraphics[width=\linewidth]{case2_4.eps}
    \caption{Trajectories of agents under exogenous system $\textbf{S}_1$}
    \label{fig:3-1}
  \end{minipage}\hfill
  \begin{minipage}{0.49\linewidth}
    \centering
    \includegraphics[width=\linewidth]{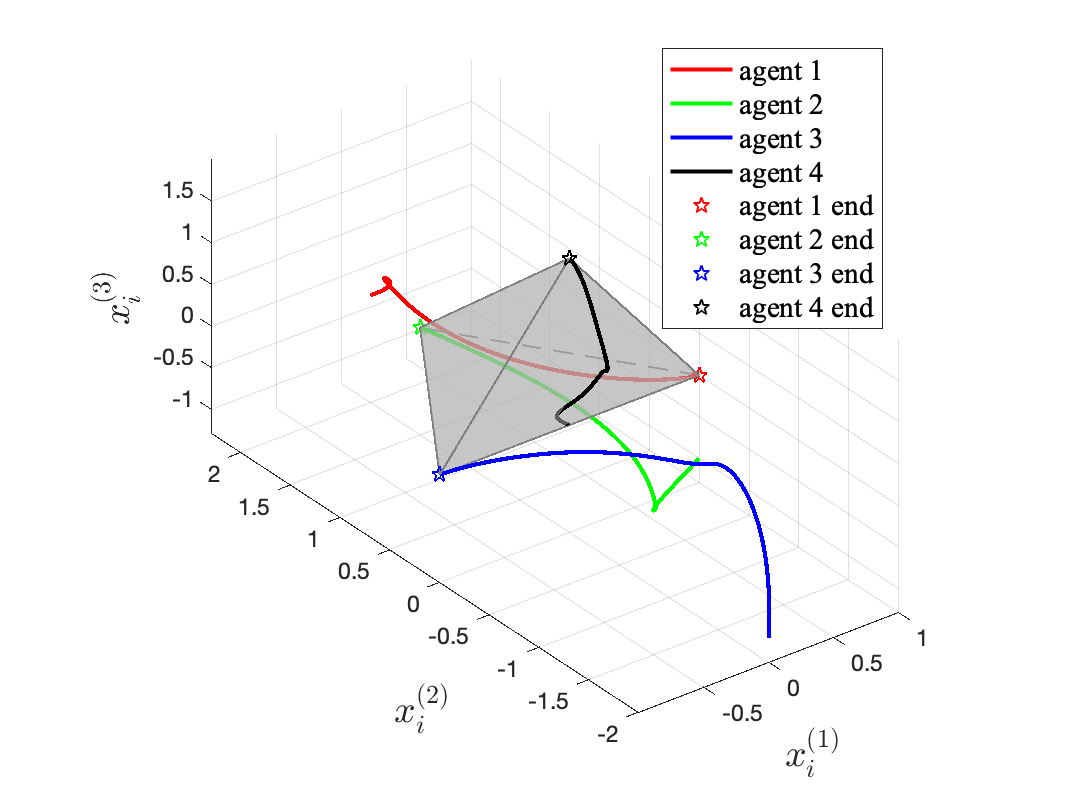}
    \caption{Trajectories of agents under exogenous system $\textbf{S}_2$}
    \label{fig:3-2}
  \end{minipage}
  \label{fig:case1}
\end{figure}
In \textbf{Case 3}, we use different exogenous systems to achieve the same formation.
In the tetrahedron formation example of Case 2, we obtained the exogenous system $\textbf{S}_1$ following Algorithm 1: 
$$H =\begin{bmatrix} 1 & 0 & 0 & 0 & 0 \\ 0 & 1 & 0 & 0 & 0 \end{bmatrix}^{\top},
\,\,
[F_1~~ F_2] =[0_{2\times 12} \; 0_{2\times 2}],$$
$$K =
\begin{bmatrix}
4 & 11 & 16 & 1 & 6 \\
1 & 4 & 2 & 8 & 12 \\
16 & 13 & 6 & 2 & 20 \\
4 & 10 & 4 & 2 & 9 \\
8 & 3 & 11 & 16 & 14
\end{bmatrix},
\,\,
G =
\begin{bmatrix}
15 & 13 & 19 & 5 & 10 \\
9 & 2 & 3 & 16 & 16
\end{bmatrix},$$
where the choices of $K\text{ and }G$ are not unique as long as they ensure that the system is stabilizable.
We can also achieve the same formation using a different exogenous system $\textbf{S}_2$:
$$H =\begin{bmatrix} 1.1 & -0.8 & 0 & 0 & 0 \\ 0.17 & 0.71 & 0 & 0 & 0 \end{bmatrix}^{\top},
\,\,
[F_1~~ F_2] =[0_{2\times 12} \; 0_{2\times 2}],$$
\small
$$\setlength{\arraycolsep}{1pt} 
K =
\begin{bmatrix}
-18 & -14 & -10 & -13 & -6 \\
  6 &   2 &  14 &   3 & -7 \\
 -3 &  18 &  18 & -13 & -2 \\
 17 &   7 &  -5 &  -9 & -11 \\
 -3 &  -4 & -18 &  -2 & -19
\end{bmatrix},
G =\begin{bmatrix}
  6 &  12 &  -7 &  10 & -16 \\
-14 &  -4 &  10 & -13 & -13
\end{bmatrix}.$$
\normalsize
Exogenous systems $\textbf{S}_1$ and $\textbf{S}_2$ employ different control strategies for system \eqref{example1} because the solutions of the Riccati equation $P$ are different. 
To achieve the same formation scaling, the initial conditions for the two exogenous systems are distinct.
Specifically, for $\textbf{S}_1$, $w(0)\approx \left[\begin{array}{llllllllllll}-2.3559 & 0\end{array}\right]^{\top}$, while for $\textbf{S}_2$, $w(0)\approx \left[\begin{array}{llllllllllll}-0.6682&0\end{array}\right]^{\top}$.
Then for sufficiently large \(t\), we obtain
$\hat{x}\approx\left[\begin{array}{llllllllllll}1 & 0 & 0 & 0 & 1.5 & 0 & -1 & 0 & 0 & 0 ~ 0 ~ 2\end{array}\right]^{\top}$ $=x_{df}^3$ under the control generated by the exogenous system $\textbf{S}_2$, which is the same as the tetrahedron formation in Case 2 with exogenous system $\textbf{S}_1$. 
Fig. \ref{fig:3-1} and Fig. \ref{fig:3-2} show the trajectories of the agents under exogenous system $\textbf{S}_1$ and $\textbf{S}_2$, respectively. These results demonstrate that by employing different control strategies, the agents follow distinct paths to reach the same formation.

\section{Conclusion}
In this paper, we study the intrinsic formation problem of multi-agent systems by introducing an exogenous system within an infinite time-horizon linear quadratic optimal control framework. By formulating the problem without explicitly incorporating the desired formation into the performance index or controller, we enhance the security of the target formation. For the forward problem, we identify the existence condition of a nonzero steady state and characterize the steady-state space. For the inverse problem, we design the input matrix $B$, exogenous system matrices and performance index to achieve the desired formation while minimizing the number of control and the dimension of the exogenous system. We also design the initial state of the exogenous system to achieve the desired formation scaling.
Numerical simulations validate the effectiveness of our methods. Our approach extends intrinsic control methods to the realization of almost arbitrary formations in any dimension—with the sole exception of eigenvectors of $A$ corresponding to eigenvalues with positive real parts and those associated with the zero eigenvalue having Jordan blocks of size at least 2—thereby overcoming limitations of previous studies confined to limited patterns.

In the future, we plan to extend our framework to more complex agent dynamics and incorporate uncertainties to enhance robustness, including exploration of nonlinear behaviors and higher-order systems. Besides, we can also adopt distributed control strategies which represent a promising approach to improve scalability and efficiency, especially in large-scale multi-agent networks. Finally, applying differential game methods could broaden the applicability of our approach, enabling the design of control laws in competitive environments.

\Acknowledgements{This work was supported in part by the National Natural Science Foundation of China with No. 62373245, in part by the National Key  Research and Development Program of China with No. 2023YFB4706800, and in part by the Dawn Program of Shanghai Education Commission, China.}

%
%
%


	
\end{document}